\documentclass[11pt]{amsart}

\usepackage{amsmath,amssymb,amsfonts}
\usepackage{esint, cancel}
\usepackage{amsthm}
\usepackage{mathrsfs}
\usepackage{comment}
\usepackage{geometry}
\usepackage{verbatim}
\usepackage{graphicx}
\usepackage{color}
\usepackage{bbm}
\usepackage{a4wide}

\usepackage[colorlinks=true,urlcolor=blue, citecolor=red,linkcolor=blue,
linktocpage,pdfpagelabels, bookmarksnumbered,bookmarksopen]{hyperref}

\newtheorem{theorem}{Theorem}[section]
\newtheorem{lemma}[theorem]{Lemma}
\newtheorem{proposition}[theorem]{Proposition}
\newtheorem{corollary}[theorem]{Corollary}

\theoremstyle{definition}
\newtheorem{remark}[theorem]{Remark}

\newtheorem*{ack}{Acknowledgments}

\newcommand{\MMM}{\color{black}}
\newcommand{\KKK}{\color{black}}

\numberwithin{equation}{section}
\def\eps{\varepsilon}

\author{Francesco Bozzola}
\address[F.\ Bozzola]{DIME Dipartimento di ingegneria meccanica, energetica, gestionale e dei trasporti 
	\newline\indent
	Universit\`a di Genova
	\newline\indent
	via alla Opera Pia 15, 16145 Genova, Italy}
\email{francesco.bozzola@edu.unige.it}

\author{Edoardo Mainini}
\address[E.\ Mainini]{DIME Dipartimento di ingegneria meccanica, energetica, gestionale e dei trasporti 
	\newline\indent
	Universit\`a di Genova
	\newline\indent
	via alla Opera Pia 15, 16145 Genova, Italy}
\email{edoardo.mainini@unige.it}

\date{\today}

\keywords{Aggregation-diffusion equations, nonlinear Riesz potential, stationary states}
\subjclass[2010]{35K44, 35R11, 49K20}

\title[Equilibria of aggregation-diffusion models with nonlinear potentials]{Equilibria of aggregation-diffusion models with nonlinear potentials}

\begin{document}
	
	\begin{abstract}
		We consider an evolution model with nonlinear diffusion of porous medium type in competition with a nonlocal drift term favoring mass aggregation. The distinguishing trait of the model is the choice of a nonlinear $(s,p)$ Riesz potential for describing the overall aggregation effect.  We investigate radial stationary states of the dynamics, showing their relation with extremals of suitable Hardy-Littlewood-Sobolev inequalities.  In the case that aggregation does not dominate over diffusion, radial stationary states also relate to global minimizers of a homogeneous free energy functional featuring the $(s,p)$ energy associated to the nonlinear potential. In the limit as the fractional parameter $s$ tends to zero, the nonlocal interaction term becomes a backward diffusion and we describe the asymptotic behavior of the stationary states.   
	\end{abstract}
	
	\maketitle
	\begin{center}
		\begin{minipage}{12cm}
			\small
			\tableofcontents
		\end{minipage}
	\end{center}
	
	\section{Introduction}
	
	We are interested in stationary solutions of aggregation-diffusion models of the form
	\begin{equation}\label{evolution}
	\partial_t \rho=\Delta\rho^m-\chi\,\mathrm{div}(\rho\,\nabla \mathcal S(\rho))\qquad\mbox{in $(0,+\infty)\times\mathbb R^N$},
	\end{equation}
	where $\rho=\rho(x,t)$ represents a mass density whose evolution is driven by a porous medium diffusion ($m>1$) and  a nonlocal interaction modeled by a potential $\mathcal S$ that accounts for long range effects. Here, $\chi>0$ is the sensitivity constant measuring 
	the interaction strength. Equations of the form \eqref{evolution} typically appear in mathematical biology as macroscopic models
	  of interacting particles/agents \cite{BCM, CFTV, ME, P}, such as the Keller-Segel model of chemotaxis \cite{GZ, H, JL, KS, KS2, N}. 
	These models usually feature linear potentials in convolution form, i.e., $\mathcal S(\rho)$ is the convolution of $\rho$ with  some suitable radial convolution kernel  accounting for mutual interaction forces. \\
	
	Among the most relevant modeling examples is the Newtonian or the Riesz (attractive) potential, appearing in the  Keller-Segel model and its many variants, which is given by
	\begin{equation}\label{rieszintro}\mathcal S(\rho)=K_s\ast\rho.\end{equation} Here, the kernel $K_s$ is defined for $0<s<N/2$ as \begin{equation}\label{rie}K_s(x) := c_{N, s}\,|x|^{2\,s-N},\qquad c_{N, s} := \pi^{-\frac{N}{2}}\,2^{-2s}\,\frac{\Gamma\left(N/2-s\right)}{\Gamma(s)},\end{equation}
	and in terms of Fourier transform we have $\hat K_s(\xi)=|\xi|^{-2s}$ (with $\hat u(\xi):=\int_{\mathbb R^N}e^{-ix\cdot\xi}u(x)\,dx$).  The particular case of the Newtonian potential corresponds to $s=1$ if $N\ge 3$. With the choice \eqref{rieszintro},
 the  free energy of the system  is 
\begin{equation*}
		\mathcal{F}_{s}(\rho) = \frac{1}{m-1} \int_{\mathbb{R}^N} \rho^m\, - \frac\chi2\int_{\mathbb R^N}\rho\,K_s\ast\rho,
	\end{equation*}
	featuring the competition among the diffusion term and the total interaction energy associated to the mean field potential.  Functional $\mathcal F_s$ has to be analyzed among
	  mass densities in the following class (defined for any given $M>0$)
	\begin{equation}\label{YM}
	\mathcal{Y}_M = \mathcal{Y}_{M, m} := \left\{\rho \in L^1_+(\mathbb{R}^N) \cap L^m(\mathbb{R}^N)\,: \,\,\int_{\mathbb{R}^N } x \,\rho(x)\, dx = 0, \,\,  \int_{\mathbb{R}^N} \rho(x)\, dx = M\right\},
	\end{equation}
	which naturally arises by taking into account that the evolution problem is formally preserving mass, center of mass and positivity. $\mathcal F_s$ is a Lyapunov functional for the dynamics. In fact,  \eqref{evolution}-\eqref{rieszintro} can be seen as the gradient flow of $\mathcal F_s$ with respect to the square Wasserstein distance, see \cite{CCH,HMVV}. In the search for stationary solutions to the evolution problem \eqref{evolution}-\eqref{rieszintro}, it is therefore natural to look for minimizers (if existing) of $\mathcal F_s$ over $\mathcal Y_M$  and, more generally, for
	critical points satisfying suitable Euler-Lagrange equations. 
	We also stress a crucial property of functional $\mathcal F_s$, which is the homogeneity  with respect to the mass invariant dilations \begin{equation} \label{def:dilation}
		\rho^{\lambda}(x) := \lambda^N\rho(\lambda x), \qquad \mbox{  } x\in \mathbb{R}^N,\,\, \lambda > 0.
	\end{equation}
	Indeed, we have
	$$\mathcal F_s(\rho^\lambda)=\frac{\lambda^{N\,(m-1)}}{m-1}\,\int_{\mathbb R^N}\rho^m - \frac{\chi\, \lambda^{N-2s}}{2}\,\int_{\mathbb R^N}\rho\,K_s\ast\rho.$$
	As a consequence, aggregation and diffusion are in balance if $m=2-2s/N$, which is called the \emph{fair competition regime} \cite{CCH}. If $m$ is below this threshold, aggregation dominates and concentrating all the mass at a single point (that is, letting $\lambda\to+\infty$) is energetically favorable.   \\


	The classical Keller-Segel model \cite{KS} of chemotaxis, in its simplest mathematical formulation \cite{BCC, BDP, DP, JL, S}  is a fair competition model,  formally obtained by letting $N=2$, $s=1$ so that the convolution kernel is the Newtonian kernel (in dimension $2$ it is understood that $K_1(x)=-\tfrac1{2\pi}\log|x|$), and by letting the diffusion be linear $m=m_c=1$ (the diffusion term in the free energy becomes $\int_{\mathbb R^N}\rho\log\rho$). It is well known that a critical mass $M_c$ exists in such a model (whose explicit value is $8\pi/\chi$), and that global-in-time solutions for the associated Cauchy problem exist if the mass is not above the critical mass, while blow up in finite time occurs if $M>M_c$. Moreover, stationary states exist only if $M=8\pi/\chi$, see \cite{BCC2,BCM, CF}.
	The above properties of the classical Keller-Segel model generalize to fair competition models in higher dimension: it has been proven in the Newtonian potential case $s=1, N\ge 3$ in \cite{BCL} that a critical mass $M_c$ still appears for $m=2-2/N$, that its value can be written in terms of the best constant of suitable Hardy-Littlewood-Sobolev (HLS) inequalities,  and that stationary states exist only if $M=M_c$. 
	 The validity of analogous properties for more singular Riesz potentials $0<s<1$ has been shown in \cite{CCH,CCH2}, still in the fair competition regime $m=2-2s/N$. On the other hand, the diffusion dominated regime has been considered in \cite{CHMV}, and in such case stationary states exist for every choice of the mass $M>0$ and can be obtained as minimizers of $\mathcal F_s$ over $\mathcal Y_M$. 
	In the aggregation dominated regime $m<2-2s/N$ the free energy $\mathcal F_s$ is not bounded from below over $\mathcal Y_M$ (whatever the choice of $M>0$), but stationary states of the dynamics can still be obtained, as seen in \cite{CGHMV}, as solutions to the Euler-Lagrange equation associated with the free energy (see also \cite{BL} for the Newtonian case $s=1$). \\

\section{Main results}
\subsection{${{(s,p)}}$ potential, stationary states and HLS inequalities}
	In this work we shall investigate the nonlinear potential counterpart of the previous results about stationary states,
	by considering an interaction described by the nonlinear Riesz potential, which has been introduced in \cite{MH}, see also  \cite{AH}, \cite[Section 4.2]{Maz}, \cite[Section 5.4]{Mizuta} and the references therein. 
	We let \begin{equation}\label{spintro}\mathcal S=\mathcal K_{s,p},\end{equation} where  $1 < p < \infty$, $0<s\,p < N$, and $\mathcal{K}_{s, p}$ stands for the {\it nonlinear $(s,p)$ Riesz potential} given by 
	\[
	\mathcal{K}_{s,p}(\rho) := K_{s/2}\ast (K_{s/2} \ast \rho)^{p'-1}. 
	\]
	Here, $p'$ is the conjugate exponent of $p$, i.e., $1/p + 1/p' = 1$.
	The total interaction energy of the mass density $\rho$ associated to the nonlinear potential $\mathcal K_{s,p}$ (the $(s,p)$ energy) is given by
	\[
	\mathcal I_{s,p}(\rho)=\frac1{p'}\int_{\mathbb R^N}\rho\,\mathcal K_{s,p}(\rho)=\frac1{p'}\int_{\mathbb R^N}(K_{s/2}\ast\rho)^{p'},
	\] 
	where the second equality is due to Plancherel theorem, which also implies that for $\eps\to0$
	\[
	\mathcal I_{s,p}(\rho+\eps\varphi)=\mathcal I_{s,p}(\rho)+\eps\int_{\mathbb R^N}(K_{s/2}\ast\rho)^{p'-1}K_{s/2}\ast\varphi+o(\eps)= \mathcal I_{s,p}(\rho)+\eps\int_{\mathbb R^N}\mathcal K_{s,p}(\rho)\,\varphi+o(\eps)
	\]
	for every test function $\varphi$, showing that indeed $\mathcal K_{s,p}$ is the functional derivative of $\mathcal I_{s,p}$. 
	The free energy is therefore \begin{equation*}
		\mathcal{F}_{s,p}(\rho) = \frac{1}{m-1} \int_{\mathbb{R}^N} \rho^m\,dx - \frac\chi{p'}\,\mathcal I_{s,p}(\rho)
	\end{equation*}
	and the evolution equation \eqref{evolution}-\eqref{spintro} is formally its Wasserstein gradient flow.
	The composition property of $K_s$ shows that for $p=2$ we are reduced to the linear potential case: $\mathcal K_{s,2}(\rho)=K_s\ast\rho$ and $\mathcal F_{s,2}(\rho)=\mathcal F_s(\rho)$.
%
%
	For $p\neq 2$ the free energy is  still a homogeneous functional, satisfying	
	$$\mathcal F_{s,p}(\rho^\lambda)=\frac{\lambda^{N\,(m-1)}}{m-1}\,\int_{\mathbb R^N}\rho^m - \lambda^{N\,(m_c-1)}\,\frac{\chi}{p'}\,\int_{\mathbb R^N}(K_{s/2} \ast \rho)^{p'},$$
	where
	\begin{equation}\label{mcintro}
	m_c:=p'-\frac{s\,p'}{N}
	\end{equation}
	is the critical  exponent.
	Therefore, we still recognize three regimes according to the value of the diffusion exponent $m$: we are in the \emph{diffusion dominated regime} if $m>m_c$, in the \emph{fair competition regime} if $m=m_c$, and in the \emph{aggregation dominated regime} if $m<m_c$.\\
	
	We perform the analysis of stationary states of \eqref{evolution}-\eqref{spintro}. As in the linear potential case, we show that a critical mass appears only if $m=m_c$. Moreover, we show that stationary states are 
	strictly related to optimizers of the following Hardy-Littlewood-Sobolev (HLS) type inequality, stating that if
	\[m>(p^*_s)',\qquad\mbox{where}\qquad p_s^*:=\frac{N\,p}{N-s\,p},
	\]
	 there exists a constant $H>0$ such that
	\begin{equation}\label{ourHLS} 
			\|K_{s/2} \ast h\|^{p'}_{p'} \le H\,\|h\|_1^{p'\,\vartheta_0} \,\|h\|_m^{p'\,(1-\vartheta_0)} \qquad \mbox{ for every } h \in L^{1}_+(\mathbb{R}^N) \cap L^m(\mathbb{R}^N),
		\end{equation}
		where  $0 < \vartheta_0 < 1$ is given by
		\begin{equation*}
			\vartheta_0 := 1 -\frac{m'}{p^*_s} = \frac{1}{(p^*_s)'}\,\frac{m - (p^*_s)'}{m-1}.
		\end{equation*}
We shall prove existence and regularity properties of optimizers of \eqref{ourHLS}, which
  will be shown to   be	
   solutions of the nonlocal equation
\begin{equation}\label{ELintro}
\rho^{m-1}=a\left(\mathcal K_{s,p}(\rho)-\mathcal C\right)_+\qquad \mbox{in $\mathbb R^N$}
\end{equation}
for suitable values of the positive constants $a,\mathcal C$. We notice that for $p=2$, in terms of $u:=K_s\ast \rho$ the above equation becomes the fractional semilinear PDE $$(-\Delta)^su=a^{\frac1{m-1}}\left(u-\mathcal C\right)^{\frac1{m-1}}_+,$$ which is the fractional plasma equation investigated in \cite{CGHMV}. The terminology for such a semilinear equation is due to the fact that the nonlinearity in the right hand side, where $(x)_+:=\max\{x,0\}$, appears in some classical models of plasma physics \cite{Temam1,Temam2}.
 The following is our first main result, which provides the main properties of the HLS optimizers. In the case that $m\ge m_c$, these results can be translated in a statement about minimizers of the free energy $\mathcal F_{s,p}$. In this regard, a critical mass appears for $m=m_c$, given by
 \begin{equation} 
		\label{massa-critica}
		M_c:= \left(\frac{p^*_s}{\chi\,H^*_{m_c,s,p}}\right)^{\frac{N}{sp'}},
	\end{equation} 
	where $H^*_{m,s,p}$ is the best constant in \eqref{ourHLS}.
 \begin{theorem}\label{main1} Let $1 < p < \infty$,  $0<s\,p < N$, $m>(p_s^*)'$. The best constant in the HLS inequality \eqref{ourHLS} is attained. Each optimizer  is radially nonincreasing (up to translation), compactly supported, H\"older regular in $\mathbb R^N$ and smooth in the interior of its support.  It satisfies \eqref{ELintro} for suitable values of the constants $a>0,\mathcal C>0$.

\noindent
If $m> m_c$, then for each optimizer $h$  of the HLS inequality \eqref{ourHLS}   there exists a unique scaling factor $\lambda>0$ such that $h^\lambda$ is (up to translation) a minimizer of functional $\mathcal F_{s,p}$ over $\mathcal Y_{M}$, where $M=\|h\|_1$; conversely for every $M>0$  minimizers  of $\mathcal F_{s,p}$ over $\mathcal Y_{M}$ exist and are optimizers of \eqref{ourHLS}.    

\noindent
If $m=m_c$, then each optimizer $h$ of the HLS inequality \eqref{ourHLS} having mass $M_c$ is (up to translation) a minimizer of $\mathcal F_{s,p}$ over $\mathcal Y_{M_c}$ and $\mathcal F_{s,p}(h)=0$; conversely minimizers of $\mathcal F_{s,p}$ over $\mathcal Y_M$ exist if and only if $M=M_c$ and are optimizers of \eqref{ourHLS}. 
\end{theorem}

Existence of optimizers of \eqref{ourHLS} is a standard application of Riesz rearrangement inequalities along with compactness theorems for radially decreasing functions. For an optimizer, the constants $a,\mathcal C$ can be explicitly expressed, as well as the optimal dilation factor $\lambda$ in the case $m>m_c$, as seen through the proof. 
It is not difficult to check that for every $M>0$ the infimum of $\mathcal F_{s,p}$ over $\mathcal Y_{M}$ equals $-\infty$ if $(p_s^*)'<m<m_c$.  
However, an optimizer of the HLS inequality is still satisfying \eqref{ELintro}, hence  after a suitable mass invariant dilation it  satisfies the Euler-Lagrange equation
\begin{equation}\label{ELintro2}
\rho^{m-1}=\frac{m-1}{m}\left(\chi \mathcal K_{s,p}(\rho)-\mathcal Q\right)_+\qquad\mbox{in $\mathbb R^N$}
\end{equation}
associated with functional $\mathcal F_{s,p}$, where $\mathcal Q>0$ is a constant playing the role of Lagrange multiplier for the mass constraint. As such,  it is (up to translation) a radially nonincreasing stationary state for \eqref{evolution}-\eqref{spintro} as we discuss in Section \ref{sec:3}.	
	About the regularity properties in Theorem \ref{main1}, we mention that
	  boundedness of optimizers has been proved in \cite{CHMV} by a purely variational argument in the case $p=2$, $m>m_c$, which consists in the construction of a suitable bounded competitor for every unbounded candidate. Such an argument seems not straightforward in the nonlinear potential setting, therefore we prove boundedness by classical bootstrap methods, based on HLS inequalities and on \eqref{ELintro}, that are working for every $m>(p_s^*)'$.
We stress that Theorem \ref{main1} generalizes the previous results in the literature about  inequality \eqref{ourHLS}: in the case $p=2$ it is also called the Lane-Emden inequality and has been studied in \cite{CCH,CLLS}. Interestingly, other generalizations have been recently investigated in \cite{GHLM}, in relation with the Choquard equation, which still leads to radially decreasing compactly supported optimizers for suitable choices of the parameters therein.  \\

	\subsection{Asymptotic behavior of stationary states as $s\to0$}
As observed in \cite{HMVV} by considering that $K_s$ is an approximate identity for small $s$, the aggregation term can be considered as an approximation of a backward diffusion process, so that the evolution model \eqref{evolution}-\eqref{spintro} formally becomes the forward-backward diffusion equation
\[
\partial_t\rho=\Delta\rho^m-\frac{\chi}{p'}\,\Delta\rho^{p'}.
\] 
Similarly, the associated free energy $\mathcal F_{s,p}$ formally becomes, in the limit $s\to0$, the following functional featuring the competition of $L^m$ and $L^{p'}$ norms
\begin{equation}\label{limit-functional}
\mathcal F_0(\rho)=\frac1{m-1}\int_{\mathbb R^N}\rho^m-\frac\chi{p'}\int_{\mathbb R^N}\rho^{p'}.
\end{equation}
Clearly, the minimization problem for functional $\mathcal F_0$ in the class $\mathcal Y_M$ is strongly influenced by the sign of $m-p'$, which is reflected in the fact that the critical exponent $m_c$ from \eqref{mcintro} is equal to $p'$ if $s=0$. If $m<p'$, then functional $\mathcal F_{s,p}$ does not have minimizers over $\mathcal Y_M$ for every small enough $s$, and moreover  $m<(p_s^*)'$ for every small enough $s$, so that we are not in the range of parameters of Theorem \ref{main1}.  Therefore, in our second main result, which is the following, we restrict to $m\ge p'$. The result for $p=2, m>2$ is given in \cite{HMVV}. 
	
	\begin{theorem}\label{main2} Let $1 < p < \infty$ and  $0<s\,p < N$. Let $m\ge p'$, $M>0$. Let $\rho_s$ be a minimizer of $\mathcal F_{s,p}$ over $\mathcal Y_M$ for every $s\in(0, N/p)$.\\
	If $m>p'$,
	 then $\rho_s\to\rho_0$ strongly in $L^q(\mathbb R^N)$ for every $1<q<+\infty$ as $s\to 0$, where $\rho_0$ is the unique radially decreasing minimizer of $\mathcal F_0$ over $\mathcal Y_M$, which is the characteristic function of a ball.\\
	Else if $m=p'$, we have $\inf_{\mathcal Y_M} \mathcal F_s\to\inf_{\mathcal Y_M} \mathcal F_0$ as $s\to 0$. Moreover, $\rho_s\to0$ uniformly on $\mathbb R^N$ if $0<\chi<p$, and $\rho_s\to M\delta_0$ in the sense of measures if $\chi>p$.
	 \end{theorem}

	Let us conclude this section with a discussion on possible further extensions and open problems. 
	First of all, uniqueness (up to translations) of stationary states of given mass (or of optimizers of the HLS inequality \eqref{ourHLS} up to the natural scaling)  would require a further, deep analysis. It has been proved in the case $p=2$ by different methods in \cite{CCH3, CLLS, CGHMV, DYY}, and each of them could be suitable for treating the nonlinear potential case as well. The stability result of the HLS inequality in \cite{CLLS} could also be potentially generalized to $p\neq2$.
	Second, radiality of \emph{every} stationary solution to \eqref{evolution}-\eqref{spintro} is not guaranteed. Such a property has been proven in \cite{CHMV} in the linear potential case $p=2$ under some  restrictions on $m,s$ (building on the result from \cite{CHVY} for $s=1$). It remains an open problem to extend such result for the case $p\neq 2$.  It would prove that all the steady states of the dynamics are actually radially decreasing.
 Moreover, it would also be interesting to investigate stationary states of the dynamics,  meant as   solutions to  \eqref{ELintro2}, in the regime $1<m\le(p_s^*)'$: in this range radially decreasing solutions are expected to exists only for $\mathcal Q=0$ and to be smooth, positive and vanishing at infinity, since this behavior has been proven for $p=2$ in \cite{CGHMV}.

	
	\section{Preliminaries}
	\subsection{Notation and functional framework} 
	The dimension of the ambient space $\mathbb{R}^N$ will be $N \ge 1$. For $x_0 \in \mathbb{R}^N$ and $r > 0$, the symbol $B_r(x_0)$ stands for the euclidean $N-$dimensional open ball 
	\[
	B_r(x_0) = \left\{ x \in \mathbb{R}^N : |x-x_0| < r\right\}. 
	\]
	As usual, we will denote with $|\,\cdot\,|$ the $N-$dimensional Lebesgue measure. For $1 \le p \le \infty$, the standard Lebesgue spaces are denoted by $L^p_{\rm loc}$ and $L^p$, and we will use the shortcut notation $\|\,\cdot\,\|_{p}$ for the $L^p(\mathbb R^N)$ norms. For an open set $\Omega \subseteq \mathbb{R}^N$, the notation $W^{1,p}(\Omega)$ and $BV(\Omega)$ stand respectively for the usual Sobolev space and the usual space of bounded variation functions on $\Omega$. 
	We use the following notation for the H\"older spaces   
	\[
	C^{0,\alpha}(\Omega) := \left\{u \in C^0(\Omega) \cap L^\infty(\Omega): \sup_{\underset{x \neq y}{x, y \in \Omega,}} \frac{|u(x) - u(y)|}{|x-y|^\alpha} < \infty \right\}.
	\]
	We say that $u \in C^{0,\alpha}_{\rm loc}(\Omega)$ if $u \in C^{0,\alpha}(\Omega')$ for every open set $\Omega'$ that is compactly contained in $\Omega$.  In particular $C^{0,1}(\Omega)$ is the space of bounded Lipschitz functions on $\Omega$.\KKK

	For every $\rho \in L^1(\mathbb{R}^N)$ and $\lambda > 0$, the {\it mass invariant dilation of $\rho$ by factor $\lambda$} is given by
	\eqref{def:dilation}.
	Since
	\begin{equation*} \label{eqn:invariant-norm}
		\|\rho\|_{1} = \|\rho^\lambda\|_{1}, 
	\end{equation*}
	if $\rho \in \mathcal{Y}_M$ then also $\rho^\lambda \in \mathcal{Y}_M$, for every $\lambda > 0$, where $\mathcal Y_M$ is defined by \eqref{YM}.

	With an abuse of notation, we will say that a radially symmetric function $\rho \in L^1(\mathbb{R}^N)$ is nonincreasing if its radial profile is nonincreasing. The {\it radially symmetric nonincreasing rearrangement} of a function $\rho \in L^1(\mathbb{R}^N)$ will be denoted by $\rho^*$. For the precise definition and its properties, we refer the reader to \cite[Chapter 3]{LL}.  We recall that the convolution among two nonnegative radially nonincreasing functions on $\mathbb R^N$ is still radially nonincreasing on $\mathbb R^N$, see \cite{Can}. In particular, if $f$ is radially nonincreasing nonnegative, so is $K_s\ast f$. \KKK
	
	The Fourier transform of the Riesz kernel $K_s$ defined in \eqref{rie}  is given by (see for example \cite[Lemma 1, Chapter V]{Steinbook} or also \cite[Theorem 2.8]{Mizuta} and \cite[Proposition 12.10]{Taheri-book}) 
	\begin{equation*} \label{eqn:fourier-riesz}
		\hat{K_s}(\xi) = |\xi|^{-2\,s}. 
	\end{equation*}
	 Moreover, for the normalization constant $c_{N, s}$ in \eqref{rie} we have the following limiting behavior
	\begin{equation} \label{eqn:asym-behav}
		\lim_{s \to 0} \frac{c_{N,s}}{s} = \frac{\Gamma\,\left(\frac{N}{2}\right)}{\pi^{\frac{N}{2}}} = \frac{2}{N\,\omega_N}.
	\end{equation}

	\subsection{Basics on Riesz potentials}
	We now recall some facts we will need throughout the whole paper. 
	
	\begin{lemma} \label{lm:bound-l-infinito-riesz}
		Let $1 \le q < r \le \infty$,  $0<s\,q < N$ and $s\,r > N$. For every $h \in L^q(\mathbb{R}^N) \cap L^r(\mathbb{R}^N)$ we have 
		\[
		\|K_{s/2} \ast h\|_\infty \le \alpha_s\,\|h\|_q + \beta_s\,\|h\|_r,
		\]
		for some positive constant $\alpha_s = \alpha(N, q, s)$ and $\beta_s = \beta(N, r, s) > 0$. Moreover, we have 
		\begin{equation} \label{asym-bound-l-infinito-riesz1}
			\lim_{s \to 0} \frac{\alpha_s}{s} = \begin{cases}
				\pi^{-\frac{N}{2}} \Gamma\left(\frac{N}{2}\right) \left(\omega_N(q-1)\right)^{\frac{q-1}{q}} \qquad &\mbox{ if } q > 1, \\
				\\
				\pi^{-\frac{N}{2}} \Gamma\left(\frac{N}{2}\right) \qquad &\mbox{ if } q=1,
			\end{cases}
		\end{equation}
		and 
		\begin{equation} \label{asym-bound-l-infinito-riesz2}
			\lim_{s \to 0} \beta_s = 1 \qquad \mbox{ if } r = \infty.
		\end{equation}
	\end{lemma}
	\begin{proof}
		{\it Case $q>1$.} Our assumptions imply that \[
		q' > \frac{N}{N-s} \qquad \mbox{ and } \qquad r' < \frac{N}{N-s}.
		\]
		Then, for every $x \in \mathbb{R}^N$, H\"older's inequality yields 
		\begin{equation} \label{stima-holder-potenziale}
			\begin{split}
				\frac{(K_{s/2} \ast h)(x)}{c_{N, s/2}} &=  \int_{\mathbb{R}^N \setminus B_1(x)} \frac{h(y)}{|x-y|^{N-s}}\,dy + \int_{B_1(x)} \frac{h(y)}{|x-y|^{N-s}}\,dy \\
				&\le \left(\int_{\mathbb{R}^N \setminus B_1} \frac{dy}{|y|^{(N-s)\,q'}}\right)^{\frac{1}{q'}}\,\|h\|_q +  \left(\int_{B_1} \frac{dy}{|y|^{(N-s)\,r'}}\right)^{\frac{1}{r'}}\,\|h\|_r \\
				&= \left(\frac{N\,\omega_N}{(N-s)\,q' - N}\right)^{\frac{1}{q'}}\|h\|_q + \left(\frac{N\, \omega_N}{N - (N-s)\,r'}\right)^{\frac{1}{r'}}\|h\|_r , 
			\end{split}
		\end{equation}
		which gives the desired conclusion with 
		\[
		\alpha_s = \alpha\left(N,s,q\right) := c_{N,s/2} \left(\frac{N\omega_N}{(N-s)\,q' - N}\right)^{\frac{1}{q'}}, \]\[ \beta_s = \beta\left(N,s,r\right) :=c_{N, s/2} \left(\frac{N \omega_N}{N - (N-s)\,r'}\right)^{\frac{1}{r'}}.
		\]
		By recalling \eqref{eqn:asym-behav}, we get  the claimed asymptotic behaviors \eqref{asym-bound-l-infinito-riesz1}-\eqref{asym-bound-l-infinito-riesz2} for $\alpha_s$ and $\beta_s$.  
		\vskip.2cm \noindent
		{\it Case $q=1$.} We pass to the limit for $q \searrow 1$ in \eqref{stima-holder-potenziale} obtaining 
		\[
		\frac{(K_{s/2} \ast h)(x)}{c_{N, s/2}} \le  \|h\|_1 + \left(\frac{N\,\omega_N}{N - (N-s)\,r'}\right)^{\frac{1}{r'}}\|h\|_r.
		\]
		This yields our claimed estimate with $\alpha_s = c_{N, s/2} $ and $\beta_s$ as before. By recalling \eqref{eqn:asym-behav}, we eventually get the desired asymptotic behaviors. 
	\end{proof}
	
	Next we introduce the {\it Hardy-Littlewood-Sobolev type inequalities} that are crucial in this work. \KKK
	
	\begin{lemma}[Hardy-Littlewood-Sobolev type inequality] \label{lm:hls}
		Let $1 < q < \infty$ and  $0<s\,q < N$. For every $h \in L^q(\mathbb{R}^N)$, we have  
		$$
		K_{s/2} \ast h \in L^{q^*_s}(\mathbb{R}^N), \qquad \mbox{ where }\,\, q^*_s = \frac{N\,q}{N-s\,q}.
		$$
		More precisely, there exists a sharp constant $\overline{H}_s = \overline{H}(N, s, q) > 0$ such that 
		\begin{equation} \label{hls-inequality-con-asym} 
			\|K_{s/2} \ast h\|_{q^*_s} \le \overline{H}_s\,\|h\|_{q}, \qquad \mbox{ with }\,\, \limsup_{s \to 0} \overline{H}_s \le 1.
		\end{equation}
		In particular, for every $h \in L^1(\mathbb{R}^N) \cap L^m(\mathbb{R}^N)$, $m>q$, we have
		\begin{equation} \label{ineq:hls}
			\|K_{s/2} \ast h\|_{q^*_s} \le  \overline{H}_s\,\|h\|_1^{\vartheta}\,\|h\|^{1- \vartheta}_m,	\qquad \mbox{ where }\,\, \vartheta = \frac{1}{q}\frac{m-q}{m-1}.
		\end{equation}
	\end{lemma}
	\begin{proof}
		Inequality \eqref{hls-inequality-con-asym} follows by using in duality the well-known {\it Hardy-Littlewood-Sobolev Inequality} \cite[Theorem 4.3]{LL}. Indeed, with the notation therein used, if we plug the following 
		\[
		r:= q, \quad \lambda := N-s \quad \mbox{ and so } \quad p:= (q^*_s)',
		\]
		we get that
		\begin{equation*} 
			\begin{split}
				\int_{\mathbb{R}^N} \varphi\,\left(K_{s/2} \ast h\right)\,dx = c_{N, s/2}\,\iint_{\mathbb{R}^N \times \mathbb{R}^N} \frac{\varphi(x)\,h(y) }{|x-y|^{N-s}}\,dxdy
				\le C\left(N,q, s\right)\,c_{N, s/2}\,\|h\|_{q}, 
			\end{split}
		\end{equation*}
		 for every $\varphi \in L^{(q_s^*)'}(\mathbb{R}^N)$ with $\|\varphi\|_{(q^*_s)'} = 1,$ which allows to conclude. The constant $C\left(N, q, s\right)$ denotes the sharp constant of \cite[Theorem 4.3(1)]{LL} and as shown therein we have 
		\begin{equation*} \label{costante-lieb}
			\begin{split}
				\overline{H}_s &= \overline{H}\left(N, q, s\right):= C\left(N, q, s\right)\,c_{N, s/2} \\
				&\le  c_{N, s/2}\,\frac{N}{s}\,\omega_N^{(N-s)/N}\,\frac{1}{q\,(q^*_s)' }\, \left(\left(\frac{(N-s)/N}{1-1/(q^*_s)'}\right)^{(N-s)/N} + \left(\frac{(N-s)/N}{1-1/q}\right)^{(N-s)/N}\right).
			\end{split}
		\end{equation*}
		By \eqref{eqn:asym-behav}, we infer that $$\limsup_{s \to 0} \overline{H}_s \le 1.$$ Eventually, by the interpolation inequality in $L^p-$spaces, we also get \eqref{ineq:hls}.
	\end{proof}
	
	For our purposes, it will be convenient to rewrite \eqref{hls-inequality-con-asym} and \eqref{ineq:hls} with $q$ replaced by $(p_s^*)'$, given that $p$ is the nonlinear Riesz potential exponent appearing in functional $\mathcal{F}_{s,p}$. 
	It reads as follows  
	\begin{corollary} \label{cor:hls}
		Let $1 < p < \infty$ and  $0<s\,p < N$. We have
			\begin{equation} \label{ineq:hls-particolare}
			\|K_{s/2} \ast h\|_{p'} \le H_s\,\|h\|_{(p^*_s)'} \quad \mbox{ for every } h \in L^{(p^*_s)'}(\mathbb{R}^N),
		\end{equation} 
		where the sharp constant $H_s  = H\left(N, s, p\right) > 0$ satisfies
		\begin{equation} \label{bound-costante-hls}
			H_s \le c_{N, s/2}\,\frac{N}{s}\,\omega_N^{(N-s)/N}\,\frac{1}{p\,(p^*_s)' }\, \left(\left(\frac{(N-s)/N}{1-1/p}\right)^{(N-s)/N} + \left(\frac{(N-s)/N}{1-1/(p^*_s)'}\right)^{(N-s)/N}\right),
		\end{equation} 
		in particular
			$\limsup_{s \to 0} H_s \le 1.$
		Moreover, if $m > (p^*_s)'$ we have
		\begin{equation} \label{ineq-hls-2}
			\|K_{s/2} \ast h\|_{p'} \le H_s\,\|h\|_1^{\vartheta_0} \,\|h\|_m^{1-\vartheta_0} \quad \mbox{ for every } h \in L^{1}(\mathbb{R}^N) \cap L^m(\mathbb{R}^N),
		\end{equation}
		where $0 < \vartheta_0 < 1$ is given by
		\begin{equation} \label{esponente-interpol}
			\vartheta_0 = \vartheta_0(m, N, p, s) :=  \frac{1}{(p^*_s)'}\,\frac{m - (p^*_s)'}{m-1} = 1 -\frac{m'}{p^*_s}.
		\end{equation}
	\end{corollary}
	
	\begin{proof} 
		We have 
		\[
		(p^*_s)' = \frac{N\,p}{N\,(p-1) + s\,p} \in \left(1,\,N/s\right),
		\]
		thus the exponent $q := (p^*_s)'$ satisfies the assumptions of Lemma \ref{lm:hls}. Since $q^*_s = p'$, inequality \eqref{ineq:hls} can be rewritten as \eqref{ineq:hls-particolare}
		where the sharp constant $H_s  = H\left(N, s, p\right) := \overline{H}(N,s,(p^*_s)') > 0$ satisfies \eqref{bound-costante-hls}. By the interpolation inequality in $L^p-$spaces we also get \eqref{ineq-hls-2}. 
	\end{proof}
		The following theorem is due to Kurokawa. It will be used in Section \ref{sec:asymptotics} to establish convergence results for minimizers of $\mathcal{F}_{s,p}$ as $s \to 0$, see Theorem \ref{thm:convergenza-forte-minimizzanti} and Proposition \ref{prop:gamma-convergenza}. We recall its elegant proof below. 
	\begin{theorem}[\cite{Kurokawa}] \label{thm:kurokawa}
		Let $1 < q < p$. For every $h \in L^q(\mathbb{R}^N) \cap L^p(\mathbb{R}^N)$, we have
		\[
		\lim_{s \to 0}\|K_{s/2} \ast h - h\|_{p} = 0.
		\] 
	\end{theorem}
	\begin{proof}
		Since $h \in L^p(\mathbb{R}^N)$, for every $\varepsilon > 0$ we can find $0 < \delta < 1$ such that 
		\begin{equation} \label{kurokawa1}
			\int_{\mathbb{R}^N} |h(x- y) - h(x)|^p\,dx < \varepsilon, \quad \mbox{ for } |y| < \delta,
		\end{equation}
		see for instance \cite[Proposition 17.1]{DiB}. We set 
		\[
		K_{s/2}(x) = c_{N, s/2}\,|x|^{s-N}\,1_{B_\delta}(x) + c_{N, s/2}\,|x|^{s-N}\,1_{B_\delta^c}(x) =: K_{s/2}^0(x) + K_{s/2}^\infty(x), 
		\]
		so we have 
		\begin{equation} \label{kurokawa-triangolo}
			\|K_{s/2} \ast h - h\|_p \le \|K_{s/2}^0 \ast h - h\|_p + \|K^\infty_{s/2} \ast h\|_p. 
		\end{equation}
		For the first addendum, since
		\begin{equation} \label{kurokawa2}
			\int_{B_\delta} |y|^{s-N}\,dy = N\,\omega_N\,\int_0^\delta \varrho^{s-1}\,d\varrho = N\,\omega_N\,\frac{\delta^s}{s},
		\end{equation}
		 by adding and subtracting $c_{N, s/2} \int_{B_\delta} |y|^{s-N} h(x)\,dy$, we get 
		\[
		\begin{split}
			\|K_{s/2}^0 \ast h - h\|_p \le c_{N, s/2}&\left(\int_{\mathbb{R}^N} \left|\int_{B_\delta} |y|^{s-N}\left(h(x- y) - h(x)\right)dy \right|^p dx\right)^{\frac{1}{p}}\\
			  &+ \left(c_{N, s/2} - \frac{s}{N\omega_N\delta^s}\right)\left(\int_{\mathbb{R}^N} \left|\int_{B_\delta} |y|^{s-N}h(x)dy\right|^pdx\right)^{\frac{1}{p}}.
		\end{split}
		\]
		By Minkowski's integral inequality, \eqref{kurokawa1} and again \eqref{kurokawa2} we get 
		\[
		\|K_{s/2}^0 \ast h - h\|_p \le N\,\omega_N\,\frac{\delta^s}{s}\,\left[c_{N, s/2}\,\varepsilon + \left(c_{N, s/2} - \frac{s}{N\,\omega_N\,\delta^s}\right)\|h\|_p\right].
		\] 
		By \eqref{eqn:asym-behav}, we then obtain 
		\begin{equation} \label{primo-pezzo}
			\limsup_{s \to 0} \|K_{s/2}^0 \ast h - h\|_p \le \varepsilon. 
		\end{equation}
		To estimate the second term in the right hand side of  \eqref{kurokawa-triangolo} for small $s$, we take $\beta, r$ so that 
		\begin{equation} \label{ass-kurokawa}
			0 < \beta < N\left(\frac{1}{q} - \frac{1}{p}\right) \quad \mbox{ and } \quad r^*_\beta = p,
		\end{equation}
		in particular $q < r < p$. Without loss of generality, we can assume that $0 < s < \beta$ and observe that 
		\[
		|y|^{s-N} = |y|^{(s-\beta) + (\beta - N)} \le \delta^{s-\beta}\,|y|^{\beta - N} \le \delta^{-\beta}\,|y|^{\beta - N}, \quad \mbox{ for } |y| \ge \delta,
		\]
		being $0 < \delta < 1$. This entails that 
		\[
		\begin{split}
			\|K_{s/2}^\infty \ast h\|_p = c_{N, s/2}\,\left(\int_{\mathbb{R}^N} \left|\int_{B^c_\delta} |y|^{s-N}\,h(x-y)\,dy\right|^pdx\right)^\frac{1}{p}
			\le c_{N, s/2}\,\delta^{-\beta}\,\left\|K_{\beta/2} \ast |h|\right\|_p. 
		\end{split}
		\]
		Since $h \in L^r(\mathbb{R}^N)$, by  \eqref{ass-kurokawa} and by the Hardy-Littlewood-Sobolev inequality \eqref{hls-inequality-con-asym} we then obtain that $\left\|K_{\beta/2} \ast |h|\right\|_p < \infty$ and so
		\[
		\lim_{s \to 0} \|K_{s/2}^\infty \ast h\|_p = 0,
		\] 
		by \eqref{eqn:asym-behav}. By spending this information and \eqref{primo-pezzo} in \eqref{kurokawa-triangolo}, we get the desired conclusion. 
	\end{proof}
	\section{Extremals of the HLS type inequality}

	\subsection{Existence: the Lieb-Oxford method} 
	In this section, we discuss the existence of extremals of the Hardy-Littlewood-Sobolev type inequality \eqref{ineq-hls-2}. 
	\vskip.2cm \noindent
	Let $1 < p < \infty$ and  $0<s\,p < N$. Let $m > (p^*_s)'$ and $\vartheta_0 = 1 - m'/p^*_s$. The following quantity 
	\begin{equation} \label{extremal-functional-hls}
		H^*_{m, s}:= \sup\left\{\frac{\|K_{s/2} \ast h\|_{p'}^{p'}}{\|h\|_1^{p'\vartheta_0} \|h\|^{p'(1-\vartheta_0)}_{m}}: h \in L^1(\mathbb{R}^N) \cap L^m(\mathbb{R}^N) \setminus \{0\}\right\},
	\end{equation}
 is the sharp constant in the Hardy-Littlewood-Sobolev type inequality \eqref{ineq-hls-2} raised to the power $p'$.
	By \eqref{bound-costante-hls} this quantity is finite, since we have $H^*_{m,s} \le H_s^{p'}.$
	\begin{remark} \label{rmk:invariant-by-dilations}
		The quotient definining $H^*_{m, s}$ given by
		\[
		h \mapsto \frac{\|K_{s/2} \ast h\|_{p'}^{p'}}{\|h\|_1^{p'\vartheta_0} \|h\|^{p'(1-\vartheta_0)}_{m}}, \qquad h \in  L^1(\mathbb{R}^N) \cap L^m(\mathbb{R}^N) \setminus \{0\},
		\]
		is {\it invariant} under the actions of the two families of trasformations 
		\[
		\mu \mapsto \mu\,h,\,\, \mbox{ for } \mu > 0, \qquad \mbox{ and } \qquad \lambda \mapsto \lambda^N\,h(\lambda\,x),\,\, \mbox{ for } \lambda > 0. 
		\]
		This property will be crucially exploited in Lemma \ref{lm:extramls-hls} below. 
	\end{remark}

	We now infer the existence of extremals for \eqref{extremal-functional-hls} by using a classical argument by Lieb and Oxford (see for instance \cite[Appendix A]{LiebOxford} and \cite[Theorem 2.5]{Lieb83}).   
	\begin{lemma}[Existence of extremals] \label{lm:extramls-hls}
		Let $1 < p < \infty$ and  $0<s\,p < N$. Let $m > (p^*_s)'$. There exists a radially symmetric and nonincreasing function $h_s \in L_{+}^1(\mathbb{R}^N) \cap L^m(\mathbb{R}^N)$ realizing the supremum in \eqref{extremal-functional-hls} and satisfying $\|h_{s}\|_1 = \|h_s\|_m =1$. \par 
		Moreover, every function $h_s \in L_{+}^1(\mathbb{R}^N) \cap L^m(\mathbb{R}^N)$ attaining the supremum in \eqref{extremal-functional-hls} is such that $h_s = h^*_s(\cdot - y)$, for some $y \in \mathbb{R}^N$. 
	\end{lemma}
	\begin{proof}
		{\bf Step 1: reduction to normalized radially symmetric and nonincreasing functions.} Let $(h_j)_{j \in \mathbb{N}}$ be a maximizing sequence of feasible competitors for $H^*_{m,s}$, i.e.
		\[
		\lim_{j \rightarrow \infty} \frac{\|K_{s/2} \ast h_j\|_{p'}^{p'}}{\|h_j\|_1^{p'\,\vartheta_0} \|h_j\|^{p'\,(1-\vartheta_0)}_{m}} = H^*_{m,s}, \quad \mbox{ with }\,\, h_j \in  L^1(\mathbb{R}^N) \cap L^m(\mathbb{R}^N) \setminus \{0\},\quad \mbox{ for }\,\, j \in \mathbb{N}. 
		\]
		We can assume that $h_j \ge 0$, since $K_{s/2} \ast h_j \le K_{s/2} \ast |h_j|$ pointwisely. We can further assume that 
		\begin{equation} \label{norm-ass-ex}
			\|h_j\|_1 = \|h_j\|_{m} = 1, \qquad \mbox{ for } j \in \mathbb{N}.
		\end{equation}
		This is not restrictive, since we could replace each approximant $h_j$ with a rescaled version given by  
		\[
		\widetilde{h_j}(x) = \lambda_j\,h_j(\mu_j\,x), \quad \mbox{ with }\,\, \mu_j = \left(\frac{\|h_j\|_1}{\|h_j\|_m}\right)^{\frac{1}{N} \frac{m}{m-1}},\quad \lambda_j = \frac{\mu_j^N}{\|h_j\|_1},\quad \mbox{ for }\,\, j \in \mathbb{N}.
		\]
		Indeed, we have 
		\[
		\begin{split}
			\|K_{s/2} \ast \widetilde{h_j}\|_{p'}^{p'} = \int_{\mathbb{R}^N} \Bigg|\int_{\mathbb{R}^N} |x-y|^{s-N}\,\lambda_j\,h(\mu_j\,y) dy\Bigg|^{p'}dx = \frac{\lambda^{p'}}{\mu^{s\,p'+N}} \|K_{s/2} \ast h_j\|_{p'}^{p'}, 
		\end{split}
		\]
		and 
		\[
		\|\widetilde{h_j}\|_1 = \frac{\lambda}{\mu^N}\,\|h_j\|_1,  \qquad \|\widetilde{h_j}\|_m = \frac{\lambda}{\mu^{\frac{N}{m}}}\,\|h_j\|_m,
		\]
		for every $j \in \mathbb{N}$. This yields 
		\[
		\begin{split}
			\frac{\|K_{s/2} \ast \widetilde{h_j}\|_{p'}^{p'}}{\|\widetilde{h_j}\|_1^{p'\,\vartheta_0}\|\widetilde{h_j}\|_m^{p'\,(1-\vartheta_0)}} &= \frac{\lambda_j^{p'}}{\lambda_j^{p'\,\vartheta_0} \lambda_j^{p'\,(1-\vartheta_0)}}\frac{\mu_j^{N\,p'\,\vartheta_0 + \frac{N}{m}\,p'\,(1-\vartheta_0)}}{\mu_j^{s\,p' +N}} \frac{\|K_{s/2} \ast h_j\|_{p'}^{p'}}{\|h_j\|^{p'\,\vartheta_0} \|h_j\|_m^{p'\,(1-\vartheta_0)}} \\
			&=  \frac{\|K_{s/2} \ast h_j\|_{p'}^{p'}}{\|h_j\|^{p'\,\vartheta_0} \|h_j\|_m^{p'\,(1-\vartheta_0)}},
		\end{split}
		\]
		for every $j \in \mathbb{N}$.
		\par
		Eventually, we claim that it is not restrictive to assume that $h_j$ is  radially nonincreasing, for every $j \in \mathbb{N}$. Indeed, take $\varphi_j \in L^{p}(\mathbb{R}^N)$ with $\|\varphi_j\|_{p} = 1$ such that 
		\begin{equation*} \label{facile}
			\|K_{s/2} \ast h_j\|_{p'} = \int_{\mathbb{R}^N} \varphi_j\,\left(K_{s/2} \ast h_j\right)\,dx.
		\end{equation*}
		We denote with $\varphi^*_j$ and $h_j^*$ the radially symmetric nonincreasing rearrangement of $\varphi_j$ and $h_j$, respectively. By using the Riesz's rearrangement inequality \cite[Theorem 3.7]{LL}, we get
		\[
		\begin{split}
			\|K_{s/2} \ast h_j\|_{p'} = c_{N, s/2}\iint_{\mathbb{R}^N \times \mathbb{R}^N}\frac{\varphi_j(x)\,h_j(y)}{|x-y|^{N-s}}dxdy \le c_{N, s/2}\iint_{\mathbb{R}^N \times \mathbb{R}^N} \frac{\varphi^{*}_j(x)\,h_j^{*}(y)}{|x-y|^{N-s}}dxdy.
		\end{split}
		\]
		 Since $\|\varphi_j^*\|_p=1$, passing to the supremum on the right hand side we obtain 
		$$
		\|K_{s/2} \ast h_j\|_{p'} \le \|K_{s/2} \ast h^*_j\|_{p'},
		$$
		which proves our claim. 
		\vskip.2cm \noindent
		{\bf Step 2: the supremum is achieved.} Thanks to {\it Step 1} we can assume that $h_j$ is a nonnegative radially symmetric nonincreasing function with $\|h_j\|_1 = \|h_j\|_{m} = 1$, for every $j \in \mathbb{N}$. By using spherical coordinates, by the monotonicity of $h_j$, we can infer that for every $R > 0$ we have
		\[
		\begin{split}
			1 = \|h_j\|_1 \ge N\,\omega_N\, \int_0^{R} \xi_j(r)\, r^{N-1}\,dr \ge N\,\omega_N \,\xi_j(R)\, \int_{0}^{R} r^{N-1}\, dr  = \omega_N\, \xi_j(R)\, R^N,
		\end{split}
		\]
		where $\xi_j(|x|) := h_j(x)$ is the one-dimensional radial profile of $h_j$, for every $j \in \mathbb{N}$.
		Similarly, we have
		\[
		1 = \|h_j\|^m_{m} \ge \omega_N\,\xi_j(R)^{m}\,R^N,
		\]
		that is 	
		\begin{equation} \label{prop:xi}
			\sup_{(R, \infty)}|\xi_j| \le \frac{1}{\omega_N}\,\min\left\{\frac{1}{R^N}, \frac{1}{R^{N/m}}\right\} =: \omega(R), \qquad \mbox{ for every } R > 0,\,\, j \in \mathbb{N}.
		\end{equation}
		Lebesgue's differentiation theorem  for monotone functions (see  \cite[Corollary 3.29]{AFP} for instance) entails that
		\[
		\int_{R}^{\infty} |\xi'_j(r)|\,dr \le \omega(R), \qquad \mbox{ for every } R > 0,\,\, j \in \mathbb{N}.
		\]
		By means of a diagonal argument and {\it Helly's Selection Theorem} (see \cite[Proposition 19.1c]{DiB}), we can extract a subsequence (not relabeled) of nonincreasing functions $(\xi_j)_j$ converging everywhere to a nonincreasing function $\xi$ in  $(R, \infty)$, for every rational number $R > 0$. This implies that  
		\begin{equation} \label{point-convergence-helly}
			\lim_{j \to \infty} h_j(x) = \lim_{j \rightarrow \infty} \xi_j(|x|) = \xi(|x|) =: h_{s}(x), \qquad \mbox{ for every } x \in \mathbb{R}^N \setminus \{0\}.
		\end{equation}
		By collecting the previous information, we infer that $h_{s}$ is a  radially nonincreasing  function satisfying 
		\begin{equation} \label{domina-il-minimo}
			0 \le h_{s}(x) \le \omega(|x|) =: v(x), \qquad \mbox{ for every } x \in \mathbb{R}^N \setminus \{0\}.
		\end{equation}
		Observe that
		\begin{equation} \label{funzione-dominante}
			v \in L^{q}(\mathbb{R}^N), \qquad \mbox{ for every } 1 < q < m.
		\end{equation}
		By using \eqref{norm-ass-ex} and \eqref{point-convergence-helly},  Fatou's Lemma entails that 
		\begin{equation} \label{fatou}
			\|h_s\|_1 \le 1 \qquad \mbox{ and } \qquad \|h_s\|_m \le 1,
		\end{equation}
		thus in particular $h_{s} \in L^1(\mathbb{R}^N) \cap L^m(\mathbb{R}^N).$
		Moreover, we have 
		\begin{equation} \label{pointwise-convergence}
			\left( K_{s/2} \ast h_{s} \right)(x) = \lim_{j \to \infty} \left(K_{s/2} \ast h_j \right) (x)  \le \left(K_{s/2} \ast v\right)(x), \qquad \mbox{ a.e. } x \in \mathbb{R}^N. 
		\end{equation} 
		Indeed, \MMM from \eqref{prop:xi} we get the inequality in \eqref{pointwise-convergence} and \KKK 
		\[
		\left(K_{s/2} \ast h_j\right)(x) = c_{N, s}\int_{\mathbb{R}^N}\frac{h_j(y)}{|x-y|^{N-s}}dy \le c_{N, s}\int_{\mathbb{R}^N}\frac{v(y)}{|x-y|^{N-s}}dy = \left(K_{s/2} \ast v\right)(x).
		\]
		From \eqref{funzione-dominante} and the properties of the Riesz potential (see \cite[Theorem 1, Chapter V]{Steinbook}) 
		\[
		\left(K_{s/2} \ast v\right)(x) < \infty, \qquad \mbox{ a.e. } x \in \mathbb{R}^N. 
		\]
		Then, by using  \eqref{point-convergence-helly} and Lebesgue's Dominated Convergence Theorem, we obtain \MMM the equality in \KKK \eqref{pointwise-convergence}. Observe that, by using \eqref{funzione-dominante} and Corollary \ref{cor:hls}, we have
		$	K_{s/2} \ast v \in L^{p'}(\mathbb{R}^N)$.
		By using Lebegue's Dominated Convergence Theorem, by \eqref{pointwise-convergence} and by recalling \eqref{norm-ass-ex}, we infer that
		\begin{equation*}
			H^*_{m,s} = \lim_{j \rightarrow \infty} \|K_{s/2} \ast h_j\|_{p'}^{p'} = \|K_{s/2} \ast h_s\|^{p'}_{p'} \le  \frac{\|K_{s/2} \ast h_s\|^{p'}_{p'}}{\|h_{s}\|_1^{p'\,\vartheta_0} \|h_s\|^{p'\,(1-\vartheta_0)}_m},
		\end{equation*} 
		where in the last inequality we also used  \eqref{fatou}.
		Then, the maximality of $H^*_{m,s}$ among functions in $L^1(\mathbb{R}^N) \cap L^m(\mathbb{R}^N)$ entails that 
		\[
		\|h_{s}\|_1^{p'\vartheta_0} \|h_s\|^{p'(1-\vartheta_0)}_m = 1.
		\]
		This combined with \eqref{fatou} gives that $\|h_{s}\|_1 = \|h_{s}\|_m = 1$. 
		\par 
		To complete the proof, we are only left out to prove that every other nonnegative extremal of \eqref{extremal-functional-hls} must be radially nonincreasing up to translations. Assume that $h_s \in L^1_{+}(\mathbb{R}^N) \cap L^m(\mathbb{R}^N) \setminus \{0\}$ satisfies equality in  \eqref{extremal-functional-hls}. In light of Corollary \ref{cor:hls} we know that $K_{s/2} \ast h_s \in L^{p'}(\mathbb{R}^N)$, thus we take $\varphi \in L^p(\mathbb{R}^N)$ such that $$\|K_{s/2} \ast h_s\|_{p'} = \int_{\mathbb{R}^N} \varphi\,(K_{s/2} \ast h_s)\,dx.$$  By using the Riesz's rearrangement inequality in strict form \cite[Theorem 3.9]{LL}, we get 
		\[
		\begin{split}
			\|K_{s/2} \ast h_s\|_{p'} = c_{N, s/2}\,\iint_{\mathbb{R}^N \times \mathbb{R}^N} \frac{\varphi(x)\,h_s(y)}{|x-y|^{N-s}}\,dxdy \le \iint_{\mathbb{R}^N \times \mathbb{R}^N} \frac{\varphi^{*}(x)\,h_s^{*}(y)}{|x-y|^{N-s}}dxdy
		\end{split}
		\]
		with equality holding only if $\varphi = \varphi^*(\cdot - y)$ and $h_s = h^*_s(\cdot - y)$, for some $y \in \mathbb{R}^N$. This proves our claim and ends the proof.
	\end{proof}
	\begin{remark}[Extremals in $\mathcal{Y}_M$]\label{extremalym}
		Under the assumptions of Lemma \ref{lm:extramls-hls}, we can infer that for every {\it prescribed mass} $M > 0$ there exists a nonnegative radially symmetric nonincreasing function $h_{s, M} \in \mathcal{Y}_M$ realizing the supremum in \eqref{extremal-functional-hls}. More precisely, it is obtained as
		\[
		h_{s, M} := M\,h_s,
		\] 
	where $h_s$ is an extremal of \eqref{ineq-hls-2} provided by Lemma \ref{lm:extramls-hls} (thus satisfying $\|h_s\|_1=1$) with barycenter at the origin.
	\end{remark}
	\subsection{Euler-Lagrange equation}
	The Euler-Lagrange equation satisfied by nonnegative extremals of \eqref{ineq-hls-2} is derived below. We share arguments from \cite[Theorem 3.1]{CCV}. 
	\begin{lemma} \label{lm:el-estremali-hls}
		Let $1 < p < \infty$ and   $0<s\,p < N$. Let $m > (p^*_s)'$. For every extremal  $h_0 \in L_{+}^1(\mathbb{R}^N) \cap L^m(\mathbb{R}^N)$ of the HLS \eqref{ineq-hls-2}, we have 
			\begin{equation} \label{eqn_el}
			\mathcal{A}_s\, h_0^{m-1} = \Big(\mathcal{K}_{s,p}(h_0) - \mathcal{C}_s \Big)_{+}\qquad \mbox{ in $\mathbb{R}^N$},
		\end{equation}
		where 
		\begin{equation} \label{costanti-eulero-lagrange}
			\mathcal{A}_s= \frac{m'}{p^*_s}\frac{\|K_{s/2} \ast h_0\|_{p'}^{p'}}{\|h_0\|_m^m}, \qquad \mathcal{C}_s = \left(1-\frac{m'}{p^*_s}\right)\,\frac{\|K_{s/2} \ast h_0\|_{p'}^{p'}}{\|h_0\|_1} > 0.  
		\end{equation}
		
	\end{lemma}
	\begin{proof}
		We set 
		\[
		\mathcal{G}_s(h):= H^*_{m,s}\,\|h\|_1^{p'\,\vartheta_0}\,\|h\|_{m}^{p'\,(1-\vartheta_0)} - \|K_{s/2} \ast h\|_{p'}^{p'},  \qquad \mbox{ for } h \in L_{+}^1(\mathbb{R}^N) \cap L^m(\mathbb{R}^N)\setminus\{0\},
		\]
		where $H^*_{m,s}$ and $\vartheta_0$ are respectively given by \eqref{extremal-functional-hls} and \eqref{esponente-interpol}. Let $h_0$ be as in the statement. \MMM We shall make perturbations of $h_0$ that preserve positivity. \KKK  We take $\varphi \in C^\infty_0(\mathbb{R}^N)$ and set 
		\begin{equation} \label{variazioni}
			\psi := \varphi\,h_0, \qquad \varepsilon_0:= \frac{1}{2 \|\varphi\|_\infty}, \qquad h_\varepsilon := h_0 + \varepsilon\,\psi \ge 0, \qquad \mbox{ for } 0 \le |\varepsilon| < \varepsilon_0. 
		\end{equation}
		By the  optimality \KKK  of $h_0$, we have 
		\[
		\mathcal{G}_s(h_\varepsilon) \ge \mathcal{G}(h_0) =0, \qquad \mbox{ for every } 0 \le |\varepsilon| < \varepsilon_0,
		\]
		this entails that 
		\begin{equation} \label{variazione-prima}
			\frac{\rm d}{\rm d\varepsilon} \biggr\rvert_{\varepsilon = 0} \mathcal{G}_s(h_\varepsilon) = 0.
		\end{equation}
		We have 
		\begin{equation} \label{var-prima-conto}
			\mathcal{G}_s(h_\varepsilon) = H^*_{m,s}\,\|h_0 + \varepsilon\,\psi\|_1^{p'\,\vartheta_0}\,\|h_0 + \varepsilon\,\psi\|_{m}^{p'\,(1-\vartheta_0)} - \|K_{s/2} \ast (h_0 + \varepsilon\,\psi)\|_{p'}^{p'}. 
		\end{equation}
		We expand to the first order, with respect to the variable $\varepsilon$, the three integral terms appearing in the rightmost term. For the first one, it is clear that 
		\[
		\int_{\mathbb{R}^N} (h_0+\varepsilon \psi)\,dx = \int_{\mathbb{R}^N} h_0\,dx + \varepsilon \int_{\mathbb{R}^N}\psi\,dx, \quad \mbox{ for } 0 < |\varepsilon| < \varepsilon_0.  
		\]
		For the second integral term, we have 
		\[
		(h_0 + \varepsilon\,\psi)^m - h_0^m = \varepsilon\int_{0}^1 m\,(h_0 + \varepsilon\,t\, \psi)^{m-1}\, \psi\,dt, \qquad \mbox{ a.e. in } \mathbb{R}^N,\,\, \mbox{ for }  0 < |\varepsilon| < \varepsilon_0.
		\]
		By integrating over $\mathbb{R}^N$, dividing by $\varepsilon$ and using Fubini theorem we obtain 
		\begin{equation} \label{integrale1}
			\int_{\mathbb{R}^N} \dfrac{\left(h_0 + \varepsilon\,\psi\right)^m - h_0^m}{\varepsilon}\,dx = \int_{0}^1 \mathscr{H}_\varepsilon(t)\,dt,
		\end{equation}
		where we set 
		\[
		\mathscr{H}_{\varepsilon}(t) = m\int_{\mathbb{R}^N} \left(h_0 + \varepsilon\,t\,\psi\right)^{m-1} \psi\,dx, \qquad \mbox{ for } t \in [0, 1].
		\]
		By H\"older's inequality and \eqref{variazioni}, we infer that 
		\begin{equation*}
			|\mathscr{H}_\varepsilon(t)| \le m\,\|h_0 + \varepsilon\,t\,\psi\|_m^{\frac{m}{m'}}\,\|\psi\|_m \le m\,\left(\|\psi\|_m + \varepsilon_0\,\|\psi\|_m\right)^{\frac{m}{m'}}\,\|\psi\|_m, \qquad \mbox{ for } t \in [0,1],
		\end{equation*}
		for every $0 <|\varepsilon| < \varepsilon_0$. By using Lebesgue's Dominated Convergence Theorem in \eqref{integrale1}
		\begin{equation*} \label{integrale1-final}
			\int_{\mathbb{R}^N} \left(h_0 + \varepsilon\,\psi\right)^m\,dx  = \int_{\mathbb{R}^N} h_0^m\,dx + \varepsilon \int_{\mathbb{R}^N} m\,h_0^{m-1}\psi\,dx + o(\varepsilon), \qquad \mbox{ as } \varepsilon \to 0.
		\end{equation*}
		For the third integral term, we have a.e. in $\mathbb{R}^N$ the following identity
		\[
		\left(K_{s/2} \ast \left(h_0 + \varepsilon\,\psi\right)\right)^{p'} - \left(K_{s/2} \ast h_0\right)^{p'} = \varepsilon\,p' \int_0^1 \left(K_{s/2} \ast \left(h_0 + \varepsilon\,t \psi\right)\right)^{p'-1} \left(K_{s/2} \ast \psi\right)\,dt,
		\]
		for $0 < |\varepsilon| < \varepsilon_0$.
		By integrating over $\mathbb{R}^N$, dividing by $\varepsilon$ and by Fubini's theorem 
		\begin{equation}\label{integrale2}
			\begin{split}
				\int_{\mathbb{R}^N}\frac{1}{\varepsilon}\left[\left(K_{s/2} \ast \left(h_0 + \varepsilon\,\psi\right)\right)^{p'} - \left(K_{s/2} \ast h_0\right)^{p'}\right]dx = \int_{0}^1 \mathscr{K}_\varepsilon(t) dt,
			\end{split}
		\end{equation}
		where we set 
		\[
		\mathscr{K}_\varepsilon(t) = p'\int_{\mathbb{R}^N} \left(K_{s/2} \ast \left(h_0 + \varepsilon\,t\,\psi\right)\right)^{p'-1} \left(K_{s/2} \ast \psi\right)\,dx, \qquad \mbox{ for } t \in [0, 1].
		\]
		By H\"older's inequality and the HLS-type inequality \eqref{ineq-hls-2}, we get 
		\[
		\begin{split}
			|\mathscr{K}_\varepsilon(t)| &\le p'\,\|K_{s/2} \ast (h_0+ \varepsilon\,t\,\psi)\|_{p'}^{\frac{p'}{p}}\,\|K_{s/2} \ast \psi\|_{p'} \\
			&\le p'\,H^*_{m,s}\,\|h_0 + \varepsilon\,t\,\psi\|_1^{\frac{p'}{p}\,\vartheta_0} \|h_0 + \varepsilon\,t\,\psi\|_m^{\frac{p}{p'}\,(1-\vartheta_0) } \|\psi\|_1^{\vartheta_0} \|\psi\|^{1-\vartheta_0}_m,
		\end{split}
		\] 
		where $H^*_{m,s}$ and $\vartheta_0$ are respectively given by \eqref{extremal-functional-hls} and \eqref{esponente-interpol}. By Minkowski's inequality, we further have 
		\begin{equation*}
			|\mathscr{K}_\varepsilon(t)| \le p'\,H^*_{m,s}\,\left(\|h_0\|_1 + \varepsilon_0\,\|\psi\|_1\right)^{\frac{p'}{p}\,\vartheta_0}\,\left(\|h_0\|_m + \varepsilon_0\,\|\psi\|_m\right)^{\frac{p'}{p}\,(1-\vartheta_0)}\,\|\psi\|_1^{\vartheta_0}\,\|\psi \|^{1-\vartheta_0}_m,
		\end{equation*}
		for $t \in [0,1]$, for $0 < |\varepsilon| < \varepsilon_0$. Thus we can use Lebesgue's Dominated Convergence Theorem in \eqref{integrale2}, obtaining
		\begin{equation*} \label{integrale2-final}
			\begin{split}
				\int_{\mathbb{R}^N} \left(K_{s/2} \ast \left(h_0 + \varepsilon\,\psi\right)\right)^{p'}dx 
				&= \int_{\mathbb{R}^N} \left(K_{s/2} \ast h_0\right)^{p'}dx + \varepsilon\,p'\int_{\mathbb{R}^N} \left(K_{s/2} \ast h_0\right)^{p'-1}\left(K_{s/2} \ast \psi\right)dx + o(\varepsilon) \\
				&= \int_{\mathbb{R}^N} \left(K_{s/2} \ast h_0\right)^{p'}dx + \varepsilon\,p'\,\int_{\mathbb{R}^N} K_{s/2} \ast \left(K_{s/2} \ast h_0\right)^{p'-1}\,\psi\,dx + o(\varepsilon),
			\end{split}
		\end{equation*} 
		as $\varepsilon \to 0$, where the last identity follows from Plancherel's theorem.
		By collecting the previous asymptotic expansions and by using that 
		\[
		H^*_{m,s} = \frac{\|K_{s/2} \ast h_0\|_{p'}^{p'}}{\|h_0\|_1^{p'\,\vartheta_0}\, \|h_0\|_m^{p'\,(1-\vartheta_0)}},
		\] from \eqref{var-prima-conto} we get 
		\begin{equation*} \label{var-prima-nulla}
			\begin{split}
				\frac{\rm d}{\rm d\varepsilon} \biggr\rvert_{\varepsilon = 0} \mathcal{G}_s(h_\varepsilon) =& p'\,\vartheta_0\,\frac{\|K_{s/2} \ast h_0\|_{p'}^{p'}}{\|h_0\|_1}\,\int_{\mathbb{R}^N} \psi\,dx + p'\,(1-\vartheta_0)\,\frac{\|K_{s/2} \ast h_0\|_{p'}^{p'}}{\|h_0\|_m^m}\,\int_{\mathbb{R}^N} h^{m-1}_0\,\psi\,dx \\
				&- p'\,\int_{\mathbb{R}^N} K_{s/2} \ast \left(K_{s/2} \ast h_0\right)^{p'-1}\,\psi\,dx.
			\end{split}
		\end{equation*}
		By recalling \eqref{variazioni} and \eqref{variazione-prima}, this entails  that 
		\[
		\int_{\mathbb{R}^N} \left(p'\vartheta_0\frac{\|K_{s/2} \ast h_0\|_{p'}^{p'}}{\|h_0\|_1} + p'(1-\vartheta_0)\frac{\|K_{s/2} \ast h_0\|_{p'}^{p'}}{\|h_0\|_m^m}h_0^{m-1} - p'K_{s/2} \ast (K_{s/2} \ast h_0)^{p'-1}\right)\varphi\,h_0\,dx = 0,
		\]
		for every $\varphi \in C^\infty_0(\mathbb{R}^N)$. By the positivity of $h_0$ on its support (\MMM which is either a ball or $\mathbb R^N$, as a consequence of Lemma \ref{lm:extramls-hls}\KKK) and by recalling the expression of $\vartheta_0$ given by \eqref{esponente-interpol}, we obtain 
		\begin{equation} \label{el-dentro-supp-hls}
			\frac{m'}{p^*_s}\frac{\|K_{s/2} \ast h_0\|_{p'}^{p'}}{\|h_0\|_m^m}\,h_0^{m-1} = \mathcal{K}_{s,p}(h_0) - \left(1-\frac{m'}{p^*_s}\right)\,\frac{\|K_{s/2} \ast h_0\|_{p'}^{p'}}{\|h_0\|_1}\qquad \mbox{ in ${\rm supp}(h_0)$}.  
		\end{equation}    
		\vskip.2cm \noindent
		In order to deduce a condition outside the support of $h_0$\KKK,  we take any nonnegative function $\varphi \in C^\infty_0(\mathbb{R}^N) \setminus \left\{0\right\}$ and set 
		\[
		 h_\varepsilon := h_0 + \varepsilon\,\varphi, \quad \mbox{ for } \varepsilon \ge 0.
		\]
		Since $\varphi \ge 0$ and $\varepsilon \ge 0$, we have $h_\varepsilon \ge 0$ thus by minimality of $h_0$ for $\mathcal{G}$ we infer  
		\[
		\lim_{\varepsilon \to 0^+} \frac{\mathcal{G}(h_\varepsilon) - \mathcal{G}(h_0)}{\varepsilon} \ge 0. 
		\]
		By arguing as before, we get 
		\[
		\int_{\mathbb{R}^N} \left(p'\,\vartheta_0\,\frac{\|K_{s/2} \ast h_0\|_{p'}^{p'}}{\|h_0\|_1} + p'\,(1-\vartheta_0)\,\frac{\|K_{s/2} \ast h_0\|_{p'}^{p'}}{\|h_0\|_m^m}\,h_0^{m-1} - p'\,K_{s/2} \ast (K_{s/2} \ast h_0)^{p'-1}\right)\,\varphi\,dx \ge 0,
		\] 
		for every nonnegative function $\varphi \in C^\infty_0(\mathbb{R}^N)$. This entails that 
		\begin{equation}\label{el-fuori}
			p'\,\vartheta_0\,\frac{\|K_{s/2} \ast h_0\|_{p'}^{p'}}{\|h_0\|_1} - p' K_{s/2} \ast \left(K_{s/2} \ast h_0\right)^{p'-1}(x) \ge 0\qquad \mbox{ in ${\rm supp}(h_0)^c$}.
		\end{equation}
		The proof is thereby complete, in light of \eqref{el-dentro-supp-hls} and \eqref{el-fuori}\KKK.  
	\end{proof}
	\begin{remark}
		We can express the constants $\mathcal{A}_s$ and $\mathcal{C}_s$ in \eqref{costanti-eulero-lagrange} in terms of $H^*_{m, s}$: we have 
		\[
		\mathcal{A}_s = \frac{m'}{p^*_s}\,H^*_{m, s}\,\|h\|_{m}^{p'-m}\,\left(\frac{\|h\|_m}{\|h\|_1}\right)^{p'\,\left(\frac{m'}{p^*_s} - 1\right)}, \qquad \mathcal{C}_s = \left(1-\frac{m'}{p^*_s}\right)\,H^*_{m, s}\,\|h\|_1^{p'-1}\,\left(\frac{\|h\|_m}{\|h\|_1}\right)^{\frac{p'\,m'}{p^*_s}}.
		\] 
	\end{remark}
	\subsection{Regularity properties} Next we show that any \MMM extremal of the HLS inequality \eqref{ineq-hls-2} \KKK has compact support and it is  bounded. We rely on a {\it bootstrap argument} based on the combination of the HLS inequality \eqref{hls-inequality-con-asym} and Lemma \ref{lm:bound-l-infinito-riesz}. 
	\begin{lemma}[$L^\infty-$bound and compactness of the support] \label{lm:bootstrap-hls}
		Let  $1 < p < \infty$  and  $0<s\,p < N$. Let $m > (p^*_s)'$.  \MMM For every extremal $h_s \in L_{+}^1(\mathbb{R}^N) \cap L^m(\mathbb{R}^N)$  of the HLS inequality \eqref{ineq-hls-2}\KKK, we have that $h_s \in L^\infty(\mathbb{R}^N)$. Moreover, the  support of $h_s$ is compact.
	\end{lemma}
	\begin{proof}
		We start by proving that ${\rm supp}(h_s)$ is compact. \MMM Recall that by Lemma \ref{lm:extramls-hls}, the support of  $h_s$ is either a ball or $\mathbb R^N$. \KKK  By contradiction, assume that ${\rm supp}(h_s) = \mathbb{R}^N$. Our assumptions entail that $h_s \in L^{(p^*_s)'}(\mathbb{R}^N)$. By using twice the HLS inequality \eqref{hls-inequality-con-asym}, we infer that \MMM $\mathcal{K}_{s,p}(h_s) \in L^{p^*_s}(\mathbb{R}^N)$\KKK, so in particular it vanishes at infinity. \MMM By using \eqref{eqn_el} and by \KKK recalling that  $\mathcal{C}_s > 0$ from \eqref{costanti-eulero-lagrange}, we get a contradiction. 
		\vskip.2cm \noindent
		We now prove that $h_s \in L^\infty(\mathbb{R}^N)$.
		We set $m_1 := m$ and distinguish two cases according to whether $s\,m_1 \ge N$ or $s\,m_1 < N.$
		\vskip.2cm \noindent
		{\it Case 1: $s\,m_1 \ge N$.} Since $h_s \in L^1(\mathbb{R}^N) \cap L^{m_1}(\mathbb{R}^N)$, in particular $h_s \in L^r(\mathbb{R}^N)$ for every $1 < r < N/s$. By Corollary \ref{cor:hls}, we infer that 
		\[
		K_{s/2} \ast h_s \in L^{t}(\mathbb{R}^N), \qquad \mbox{ for every } N/(N-s) < t < \infty,
		\]
		and so 
		\[
		(K_{s/2} \ast h_s)^{p'-1} \in L^{t\,(p-1)}(\mathbb{R}^N), \qquad \mbox{ for every } N/(N-s) < t < \infty.
		\] 
		Since $s\,p < N$, we have 	
		\[
		\lim_{t \to \frac{N}{N-s}} t\,(p-1) < N/s.
		\]
		These considerations entail that 
		\[
		(K_{s/2} \ast h_s)^{p'-1} \in L^{t_1}(\mathbb{R}^N) \cap L^{t_2}(\mathbb{R}^N), \qquad \mbox{ for some } t_1 < N/s \mbox{ and } t_2 > N/s.
		\] 
		By Lemma \ref{lm:bound-l-infinito-riesz}, we then obtain $K_{s/2} \ast \left(K_{s/2} \ast h_s\right)^{p'-1} \in L^\infty(\mathbb{R}^N)$. In turn, from the Euler-Lagrange equation Lemma \ref{lm:el-estremali-hls}, we conclude that $h_s \in L^\infty(\mathbb{R}^N)$. 
		\vskip.2cm \noindent
		{\it Case 2: $s\,m_1 < N$.} Since $h_s \in L^1(\mathbb{R}^N) \cap L^{m_1}(\mathbb{R}^N)$, from the HLS inequality \eqref{hls-inequality-con-asym}, we infer that $K_{s/2} \ast h_s \in L^r(\mathbb{R}^N) \cap L^{(m_1)^*_s}(\mathbb{R}^N)$, for every $N/(N-s) < r < (m_1)^*_s$. This entails that
		\begin{equation*} \label{caso-2-bootstrap}
			(K_{s/2} \ast h_s)^{p'-1} \in L^{r(p-1)}(\mathbb{R}^N) \cap L^{(m_1)^*_s(p-1)}(\mathbb{R}^N), \qquad \mbox{ for every } N/(N-s) < r < (m_1)^*_s. 
		\end{equation*}
		If $(m_1)^*_s\,(p-1) \ge N/s$, by arguing as in {\it Case 1}, we conclude that $h_s \in L^\infty(\mathbb{R}^N)$ and we stop.
		If otherwise $(m_1)^*_s\,(p-1) < N/s$, by the HLS inequality \eqref{hls-inequality-con-asym} we infer
		\[
		K_{s/2} \ast (K_{s/2} \ast h_s)^{p'-1} \in L^{((m_1)^*_s\,(p-1))^*_s}(\mathbb{R}^N).     
		\]
		In turn, from the Euler-Lagrange equation, this implies that 
		\[
		h_s \in L^{m_2}(\mathbb{R}^N), 
		\]
		where we have set \[
		m_2 := (m-1)\,((m_1)^*_s\,(p-1))^*_s = \frac{N\,(m-1)\,(p-1)\,m_1}{N-s\,p\,m_1},
		\]
		and observe that $m_2 > m_1$, being this condition equivalent to $m > (p^*_s)'.$ In general, let $k \in \mathbb{N} \setminus \{0\}$ and assume that $m_i < N/s$ for every $1 \le i \le k$. We define  
		\begin{equation} \label{def-parametri-integrabilità}
			m_{k+1} :=  (m-1)\,((m_k)^*_s\,(p-1))^*_s = \frac{N\,(m-1)\,(p-1)\,m_k}{N-s\,p\,m_k}.
		\end{equation}
		We want to prove by induction that $	m_{k+1} > m_{k}.$
		Our inductive assumption reads as 
		\begin{equation} \label{hp-ind}
			m_{i+1} > m_i, \qquad \mbox{ for every } 1 \le i \le k-1. 
		\end{equation}
		Since $m > (p^*_s)'$, we have
		\[
		m > \frac{(N-s\,m)\,p}{N\,(p-1)} = \frac{(N-s\,m_1)\,p}{N\,(p-1)} > \frac{(N-s\,m_k)\,p}{N\,(p-1)},
		\]
		where in the last inequality we used \eqref{hp-ind}. In particular
		\[
		m >  \frac{(N-s\,m_k)\,p}{N\,(p-1)} \quad \Longleftrightarrow \quad m_{k+1} > m_k,
		\]
		as we can infer by recalling \eqref{def-parametri-integrabilità}. We now claim that \begin{equation} \label{esplode}
			\lim_{k \to \infty} m_k = +\infty.
		\end{equation}
		This would entail that for some $\overline{k}$ we must have $m_{\overline{k}} \ge N/s$, thus, by arguing as in {\it Case 1}, this would also end the proof. Since $m_k \ge m$, we have
		\[
		\frac{m_{k+1}}{m_k} = \frac{N\,(m-1)\,(p-1)}{N-s\,p\,m_k} \ge \frac{N\,(m-1)\,(p-1)}{N-s\,p\,m} = \frac{m_2}{m_1}, \qquad \mbox{ for every } k \ge 1.
		\] 
		Moreover, as we have already observed 
		\[
		\frac{m_2}{m_1} > 1 \quad \Longleftrightarrow \quad m > (p^*_s)'.
		\]
		The last two facts entail that   
		\[
		\liminf_{k \to \infty} \frac{m_{k+1}}{m_k} > 1,
		\]
		and so our claim \eqref{esplode}. 
	\end{proof} 
	In the next lemma, we can readily adapt the argument of \cite[Theorem 8]{CHMV} to infer H\"older regularity for extremals of \eqref{ineq-hls-2}. \KKK 
	\begin{lemma}[H\"older regularity] \label{lm:holder}
		Let  $1 < p < \infty$  and  $0<s\,p < N$. Let $m > (p^*_s)'$. For every \MMM extremal $h_s \in L_{+}^1(\mathbb{R}^N) \cap L^m(\mathbb{R}^N)$ of the HLS inequality \eqref{ineq-hls-2}, we have 
		\begin{itemize}
			\item  for $0 < s < 1/2$ 
			\[
			\begin{cases}
				\begin{aligned}
				\MMM 	h_s \in C^{0, 1}(\mathbb{R}^N), \qquad &\mbox{ if } m \le 2,\KKK\\
					\\
					h_s \in C^{0, \frac{1}{m-1}}(\mathbb{R}^N), \qquad &\mbox{ if } 2 < m < m^*, \\
					\\
					h_s \in C^{0, \gamma}(\mathbb{R}^N), \qquad &\mbox{ if } m^* \le m,
				\end{aligned}
			\end{cases}
			\]
			where $m^* = \dfrac{2 - 2s}{1-2s}$ and $\gamma \in \left(0, \dfrac{2\,s}{m-2}\right)$. \\
			\item for $s \ge 1/2$ 
			\[
			h_s \in C^{0, \gamma}(\mathbb{R}^N), \qquad \mbox{ where } \gamma = \min\left\{1, \frac{1}{m-1}\right\}.
			\]
		\end{itemize}
		Moreover, $h_{s}$ has $C^\infty-$regularity in the interior of its support.
	\end{lemma}
	\begin{proof}
		First, we assume $0 < s < 1/2$. 
		\MMM We will take advantage of the embeddings between Bessel potential spaces, fractional Sobolev spaces and  H\"older spaces.
	\MMM We briefly recall that for $1 \le q < \infty$ the fractional Sobolev space $W^{s,q}(\mathbb{R}^N)$ is  given by 
	\[
	W^{s,q}(\mathbb{R}^N) = \left\{u \in L^q(\mathbb{R}^N) : [u]_{W^{s,q}(\mathbb{R}^N)} < \infty\right\},
	\]
	where $[\,\cdot\,]_{W^{s,q}(\mathbb{R}^N)}$ denotes the Gagliardo-Slobodecki\u{\i} seminorm 
	\[
	[u]_{W^{s,q}(\mathbb{R}^N)} := \left(\iint_{\mathbb{R}^N \times \mathbb{R}^N} \frac{|u(x) - u(y)|^q}{|x-y|^{N+s\,q}}\,dxdy\right)^{\frac{1}{q}}.
	\]
	 The Bessel potential spaces $\mathcal L^{s,q}(\mathbb R^N)$, where $1< q <\infty$, are defined through the Fourier transform, see for instance \cite[Section V.3]{Steinbook}, \cite[Section 2.2.2]{Triebel1}, \cite[Section 27.3]{Samkobook}. They	can be characterized as 
	\[
	\mathcal{L}^{s,q}(\mathbb{R}^N) = \{u \in L^q(\mathbb{R}^N): u = K_{s/2} \ast h, \mbox{ for some } h \in L^q(\mathbb{R}^N)\},
	\]
	see for example \cite[Theorem 2]{Stein61} (or also \cite[Theorem 26.8, Theorem 27.3]{Samkobook}).
		\KKK By recalling \cite[Theorem 5, pag. 155]{Steinbook} and \cite[Theorem 4.47]{Demengelbook}, the following continuous embeddings holds true: 
		\begin{align} \label{embedding1}
			\mathcal{L}^{s,q}(\mathbb{R}^N) \hookrightarrow W^{s,q}(\mathbb{R}^N), \qquad \mbox{ for } q \ge 2,
		\end{align}
		and 
		\begin{align} \label{embedding2}
			W^{s,q}(\mathbb{R}^N) \hookrightarrow C^{0, \gamma}(\mathbb{R}^N), \qquad \mbox{ where } \gamma = s-N/q, \mbox{ for } q > N/s.
		\end{align}
		Let $h_s$ be as in the statement. \MMM By Lemma \ref{lm:extramls-hls}, we can assume it is radially nonincreasing. By Lemma  \ref{lm:bound-l-infinito-riesz}, Lemma \ref{lm:el-estremali-hls} and Lemma \ref{lm:bootstrap-hls}, we have that 
		\begin{equation} \label{inclusione-fondamentale}
			h_s \in L^1(\mathbb{R}^N) \cap L^\infty(\mathbb{R}^N) \quad \mbox{ and }\quad K_{s/2} \ast h_s \in L^q(\mathbb{R}^N), \quad \mbox{for every } N/(N-s) < q \le \infty.
		\end{equation}
		In particular, since $N/(N-s) < N/s$ for every $0 < s < 1/2$, we have that   
		\[
		K_{s/2} \ast h_s \in \mathcal{L}^{s,q}(\mathbb{R}^N), \qquad \mbox{ for every } q > N/s.
		\]
		In view of the embeddings \eqref{embedding1} and \eqref{embedding2}, this entails that 
		\begin{equation*} \label{boot-stima-prelim0}
			K_{s/2} \ast h_s \in C^{0, \gamma}(\mathbb{R}^N), \qquad \mbox{ where } \gamma = s - N/q, \mbox{ for } q > N/s.
		\end{equation*}
		 Since $h_s$ is radially nonincreasing, so is of $K_{s/2} \ast h_s$, see \cite{Can}. Moreover, $K_{s/2} \ast h_s$ is clearly positive, bounded and vanishing at infinity. In particular it is bounded away from zero on compact sets. For these reasons we get 
		 $(K_{s/2} \ast h_s)^{p'-1} \in C^{0, \gamma}_{\rm loc}(\mathbb R^N)$. Therefore for every $R>0$, by using \eqref{inclusione-fondamentale}, Lemma \ref{lm:bound-l-infinito-riesz} and the identity  \KKK
		\begin{equation*} \label{eqn:potenziale-frac-lapl}
			(-\Delta)^{s/2}\left(K_{s/2} \ast (K_{s/2} \ast h_s)^{p'-1}\right) = (K_{s/2} \ast h_s)^{p'-1} \qquad \mbox{ in } B_{2R},
		\end{equation*} 
		from \cite[Corollary 3.5]{RosOtonSer} we can infer that  
		\begin{equation}  \label{schauder}
			\begin{split}
				\|K_{s/2} \ast (K_{s/2} \ast &h_s)^{p'-1}\|_{C^{0, \gamma + s}(B_{R})} \le \\ &\le c \left(\|K_{s/2} \ast (K_{s/2} \ast h_s)^{p'-1}\|_{L^\infty(\mathbb{R}^N)} + \|(K_{s/2} \ast h_s)^{p'-1}\|_{C^{0, \gamma}(B_{2R})}\right),
			\end{split}
		\end{equation}
		for some $c = c(N,s, R) > 0$ (notice that since $s<1/2$, then $\gamma+s$ is not an integer as $\gamma=s-N/q$\KKK) and so
		\[
		K_{s/2} \ast (K_{s/2} \ast h_s)^{p'-1} \in C^{0, \gamma+s}(B_{R}), \qquad \mbox{ where } \gamma = s - N/q, \mbox{ for } q > N/s.
		\]
		From the Euler-Lagrange equation provided by Lemma \ref{lm:el-estremali-hls}, this entails that $h_s^{m-1} \in C^{0, \gamma + s}(B_R)$, for $R > 0$. By using the fact that $h_s$ has compact support, we infer 
		\begin{equation} \label{boot-stima-prelim-2}
			h_s \in C^{0, (\gamma + s)\alpha}(\mathbb{R}^N), \quad \mbox{ where } \alpha = \min\left\{1, \frac{1}{m-1}\right\} \mbox{ and } \gamma = s - N/q, \mbox{ for } q > N/s.
		\end{equation}
		Now we distinguish three cases.
		\vskip.2cm \noindent
		{\it Case 1: $\MMM (p_s^*)'\KKK < m \le 2.$} By using \eqref{inclusione-fondamentale} and \eqref{boot-stima-prelim-2}, from \cite[Corollary 3.5]{RosOtonSer} we have that 
		\[
		K_{s/2} \ast h_s \in C^{0, \gamma + 2s}_{\rm loc}(\mathbb{R}^N),
		\]
		if $\gamma + 2s$ is not an integer.  Since $K_{s/2} \ast h_s$  is bounded and bounded away from zero on compact sets, as before we deduce\KKK  
		\[
		(K_{s/2} \ast h_s)^{p'-1} \in C^{0, \gamma + 2s}_{\rm loc}(\mathbb{R}^N).
		\]
		If $\gamma + 2s > 1$, we get $(K_{s/2} \ast h_s)^{p'-1} \in C^{0,1}_{\rm loc}(\mathbb{R}^N)$ and so also $K_{s/2} \ast (K_{s/2} \ast h_s)^{p'-1} \in C^{0, 1}_{\rm loc}(\mathbb{R}^N)$. Thus, in light of Lemma \ref{lm:el-estremali-hls}, using that $m\le 2$ and the compactness of the support of $h_s$, \KKK we get $h_s \in C^{0, 1}(\mathbb{R}^N)$ as desired.
		On the other hand,  if $\gamma + 2s < 1$ we newly apply \eqref{schauder} and \cite[Corollary 3.5]{RosOtonSer} obtaining that
		\[
		K_{s/2} \ast (K_{s/2} \ast h_s)^{p'-1} \in C^{0, \gamma + 3s}_{\rm loc}(\mathbb{R}^N),
		\]
		if $\gamma + 3\,s$ is not an integer. Observe that we gained $2s$ derivatives starting from \eqref{boot-stima-prelim-2}, \MMM and the gain in regularity depends therefore on $s$ but not on $p$. In other words, the regularity gain provided by the nonlinear potential $\mathcal K_{s,p}$ does not depend on $p$ and it is the same of the linear potential $\mathcal K_{s,2}$. Thus, the proof gets reduced to the case $p=2$ which is given in \cite[Theorem 8]{CHMV}.  For this reason, we just sketch the conclusion of the argument, omitting some details. \KKK We take an integer $j \ge 1$ such that 
		\begin{equation} \label{scelta_h}
			\frac{1}{2\,(j+1)} < s < \frac{1}{2\,j},
		\end{equation}
		and set 
		$
		\gamma_j := \gamma + (j-1)\,2\,s = 2\,s\,j - {N}/{q}, 
		$
		where $q > N/s$ is choosen large enough so that  $
			1-2\,s <\gamma_j < 1.$
		This is a feasible choice thanks to \eqref{scelta_h}. By iterating the previous argument $j$ times starting from \eqref{boot-stima-prelim-2}, we get  $K_{s/2} \ast (K_{s/2} \ast h_s)^{p'-1} \in C^{0, \gamma_j + 2s}_{\rm loc}(\mathbb{R}^N) \subseteq C^{0, 1}_{\rm loc}(\mathbb{R}^N),$ being $\gamma_j + 2s > 1$ by construction. By using the Euler-Lagrange equation provided by Lemma \ref{lm:el-estremali-hls} and the compactness of ${\rm supp}(h_s)$ from Lemma \ref{lm:bootstrap-hls} and since $m\le 2$, we conclude that
		$h_s \in C^{0,1}(\mathbb{R}^N)$. 
		\vskip.2cm \noindent
		{\it Case 2: $2 < m < m^*$.}
		Starting from \eqref{boot-stima-prelim-2}, we can improve the H\"older regularity of $h_s$ by a { bootstrap argument}, as in the previous case. We give the details of one iteration, in order to clarify that the same argument used in the proof of \cite[Theorem 8]{CHMV} still holds. 
		In light of \eqref{inclusione-fondamentale} and \eqref{boot-stima-prelim-2}, we can apply \cite[Corollary 3.5]{RosOtonSer} to infer that  
		\[
		(K_{s/2} \ast h_s)^{p'-1} \in C^{0, \frac{\gamma +s}{m-1} + s}_{\rm loc}(\mathbb{R}^N), \qquad \mbox{ where } \gamma = s - N/q, \mbox{ for } q > N/s,
		\]
		if $(\gamma +s)/(m-1) + s$ is not an integer. By newly applying \cite[Corollary 3.5]{RosOtonSer}, we obtain \[
		K_{s/2} \ast (K_{s/2} \ast h_s)^{p'-1} \in C^{0, \frac{\gamma +s}{m-1} + 2\,s}_{\rm loc}(\mathbb{R}^N),
		\]
		if $(\gamma +s)/(m-1) + 2\,s$ is not an integer. By reasoning as in the previous case, if $\gamma + 2\,s/(m-1) + 2\,s > 1$, we have $h^{m-1}_s \in C^{0, 1}(\mathbb{R}^N)$ and so  $h_s \in C^{0, \frac{1}{m-1}}(\mathbb{R}^N)$. On the other hand, if $\gamma + 2\,s/(m-1) + 2\,s < 1$, by always using the Euler-Lagrange equation provided by Lemma \ref{lm:el-estremali-hls} and the fact that $h_s$ has compact support, we obtain 
		$
		h_s^{m-1} \in C^{0, \frac{\gamma +s}{m-1} + 2\,s}(\mathbb{R}^N),
		$
		which entails that 
		\[
		h_s \in C^{0, \gamma_1}(\mathbb{R}^N), \qquad \mbox{ where }  \gamma_1 = \frac{\gamma +s}{(m-1)^2} + \frac{2\,s}{m-1},  
		\]
		if $(\gamma +s)/(m-1) + 2\,s$ is not an integer, where $\gamma = s - N/q$, for  $q > N/s$. In general, by iterating this argument following \cite[Theorem 8]{CHMV}, we can improve the H\"older regularity of $h_s$ to infer that $h^{m-1}_s \in C^{0, 1}(\mathbb{R}^N)$, which yields \KKK $h_s \in C^{0, \frac{1}{m-1}}(\mathbb{R}^N),$ as desired. 
		
		\vskip.2cm \noindent
		{\it Case 3: $m \ge m^*.$}
		\MMM We can proceed with the same bootstrap argument, however without reaching Lipschitz regularity of $h_s^{m-1}$. \KKK We observe that \cite[Remark 2]{CHMV}, to which we refer, still holds and gives the desired result. 
		
		\vskip.2cm 
		
		In order to conclude, we observe that the case $1/2 \le s $ is simpler than the case $0 < s < 1/2$ and can be treated in the same way, up to some minor modifications, as done in \cite[Theorem 8]{CHMV}. Eventually, by using the same argument of \cite[Theorem 10]{CHMV}, from Lemma \ref{lm:el-estremali-hls} and Lemma \ref{lm:bootstrap-hls}, we obtain the $C^\infty-$regularity of $h_{s}$ in the interior of its support.
	\end{proof}
	We end this section by remarking that the extremals of \eqref{extremal-functional-hls} are always in $W^{1,1}(\mathbb R^N)$: 
	\begin{corollary}\label{W11}
		Let  $1 < p < \infty$  and  $0<s\,p < N$. Let $m > (p^*_s)'$. For every  function $h_s \in L^1_+(\mathbb{R}^N) \cap L^m(\mathbb{R}^N)$ attaining the supremum in \eqref{extremal-functional-hls}, we have $h_{s} \in W^{1,1}(\mathbb{R}^N)$.
	\end{corollary}
	\begin{proof}
		The desired conclusion follows from Lemma \ref{lm:extramls-hls}, Lemma \ref{lm:el-estremali-hls} and Lemma \ref{lm:bootstrap-hls} by arguing as in \cite[Proposition 2.10]{HMVV}. 
	\end{proof}
	\section{Minimizers of the energy functional } \label{sec:3}
	In this section, we analyze minimizers of functional $\mathcal{F}_{s,p}$ over $\mathcal Y_M$ thus concluding the proof of Theorem \ref{main1}. We start by considering the {\it diffusion dominated regime}, that is the case $m > m_c$.
	\begin{proposition}[Diffusion dominated regime] \label{prop:dom-regime}
	Let $1 < p < \infty$ and  $0<s\,p < N$. Let $m > m_c$ and let $\chi, M > 0$. Define the functional
		\begin{equation*} \label{funzionale-ausiliario}
			\mathcal{Y}_M \ni \rho \mapsto \Lambda(\rho): =  \left(\frac{\| K_{s/2} \ast \rho \|_{p'}^{p'\,(m-1)}}{\|\rho\|_m^{m\,(m_c-1)}} \right)^{\frac{1}{m-m_c}},		\end{equation*}
		which is invariant by mass-invariant dilations. Let moreover
		\[\kappa: = \left(\frac{\chi}{p'}\right)^{\frac{m-1}{m-m_c}} \left(\frac{p'}{p^*_s}\right)^{\frac{m_c-1}{m-m_c}} \left(\frac{m_c - m}{m-1}\right). 
\]
		 For every  $\rho \in \mathcal{Y}_M$, there exists a unique positive number $\lambda_{*}(\rho)$, called the optimal dilation factor of $\rho$, such that 
		\begin{equation*} \label{minimality}
			\mathcal{F}_{s,p}(\rho^\lambda) \ge \mathcal{F}_{s,p}(\rho^{\lambda_{*}(\rho)}) = \kappa\,\Lambda(\rho) \quad \mbox{for every } \lambda > 0, \qquad\mbox{with  equality only if $\lambda = \lambda_{*}(\rho)$.}
		\end{equation*}
		It is expressed as
		\begin{equation} \label{fattore-dilatazione-rho00}
			\lambda_{*}(\rho) = \left(\frac{\chi}{p^*_s}\,\frac{ \|K_{s/2} \ast \rho\|_{p'}^{p'} }{\|\rho\|_m^m }\right)^{\frac{1}{N\,(m-m_c)}}.
		\end{equation}

	\end{proposition}
	\begin{proof}
		Let $\rho \in \mathcal{Y}_M$. For $\lambda  >0$, we consider the function given by 
		\begin{equation} \label{defi:functional-lambda}
			\lambda \mapsto f_{\rho}(\lambda) := \mathcal{F}_{s,p}(\rho^\lambda) =  \frac{\lambda^{N\,(m-1)}}{m-1}\,\|\rho\|_m^m - \lambda^{N\,(m_c-1)}\,\frac{\chi}{p'}\,\|K_{s/2} \ast \rho\|_{p'}^{p'}.
		\end{equation}
		Recall that, since $1 < p < N/s$, we have $N/(N-s) <p' < \infty,$ thus $m_c = p'(1-s/N) > 1.$ By optimizing with respect to $\lambda$, we get 
		\begin{equation}\label{criticality}
		\frac{{\rm d}}{{\rm d}\lambda} \mathcal{F}(\rho^\lambda) = N\,\lambda^{N\,(m-1) - 1}\,\|\rho\|_m^m - \frac{ N\,\left(m_c-1\right)\,\chi}{p'} \lambda^{N\,(m_c-1)-1}\,\|K_{s/2} \ast \rho\|_{p'}^{p'} = 0.
		\end{equation}
		The unique extremal is given by \eqref{fattore-dilatazione-rho00}. \KKK
		 Clearly, at $\lambda_{*}(\rho)$, the function given by  \eqref{defi:functional-lambda} attains a global minimum, and  notice also that we \KKK have
		\begin{equation*} \label{eqn:limiti-lambda}
			\lim_{\lambda \to +\infty} \mathcal{F}_{s,p}(\rho^\lambda) = +\infty \qquad \mbox{ and } \qquad \lim_{\lambda \to 0} \mathcal{F}_{s,p}(\rho^\lambda) = 0. 
		\end{equation*}
		Furthermore,  by using \eqref{fattore-dilatazione-rho00} \KKK we can write
		\begin{equation*} \label{eqn:conto1}
			\begin{split}
				\mathcal{F}_{s,p}(\rho^{\lambda_{*}(\rho)}) =& \frac{1}{m-1} \left(\frac{\chi}{p^*_s}\frac{\|K_{s/2} \ast \rho\|_{p'}^{p'}}{\|\rho\|_m^m}\right)^{\frac{m-1}{m-m_c}} \|\rho\|^m_m - \frac{\chi}{p'} \left(\frac{\chi}{p^*_s}\frac{\|K_{s/2} \ast \rho\|_{p'}^{p'}}{\|\rho\|_m^m}\right)^{\frac{m_c-1}{m-m_c}} \|K_{s/2} \ast \rho\|_{p'}^{p'} \\
				=& \kappa\,\Lambda(\rho)<0\KKK,
			\end{split}
		\end{equation*}
		where $\Lambda$ and $\kappa$ are defined as in the statement.
	\end{proof}
	

	For the {\it fair competition regime}, that is $m = m_c$, we have the following two sided-estimate for the energy, which extends \cite[Proposition 3.4]{BCL}. 
	\begin{proposition}[Fair competition regime] \label{prop:fair-comp}
		Let $1 < p < \infty$ and  $0<s\,p < N$. Let $\chi, M > 0$. For every $\rho \in \mathcal{Y}_M$, we have 
		\begin{equation*} \label{stima-doppia}
			\frac{\chi}{p'}\,H^*_{m_c,s}\,\left(M_c^{p' \frac{s}{N}} - M^{p'\frac{s}{N}}\right) \|\rho\|_{m_c}^{m_c} \le \mathcal{F}_{s,p}(\rho) \le \frac{\chi}{p'}\,H^*_{m_c,s}\,\left(M_c^{p' \frac{s}{N}} + M^{p'\frac{s}{N}}\right) \|\rho\|_{m_c}^{m_c},
		\end{equation*} 
		where $H^*_{m_c,s}$ is given by \eqref{extremal-functional-hls} and $M_c$ is given by \eqref{massa-critica}.\KKK
	\end{proposition}
	\begin{proof}
		For every $\rho \in \mathcal{Y}_M$, by recalling \eqref{extremal-functional-hls}, we get 
		\[
		\begin{split}
			\mathcal{F}_{s,p}(\rho) = \frac{\|\rho\|^{m_c}_{m_c}}{m_c-1} - \frac{\chi}{p'} \|K_{s/2} \ast \rho\|_{p'}^{p'} &\ge \left(\frac{1}{m_c - 1} - \frac{\chi}{p'}\,H^*_{m_c,s}\,M^{p' \frac{s}{N}}\right)\|\rho\|^{m_c}_{m_c} \\
			&= \frac{\chi}{p'}\,H^*_{m_c,s}\left(M_c^{p' \frac{s}{N}} - M^{p'\frac{s}{N}}\right) \|\rho\|_{m_c}^{m_c},
		\end{split}
		\]
		where $M_c$ is given by \eqref{massa-critica}. On the other hand, by recalling \eqref{mcintro}, we also have  
		\[
		\mathcal{F}_{s,p}(\rho) \le \frac{\chi}{p'}\,H^*_{m_c,s}\left(M_c^{p' \frac{s}{N}} + M^{p'\frac{s}{N}}\right) \|\rho\|_{m_c}^{m_c},
		\]
		which yields the claimed estimate.
	\end{proof}

	\begin{proposition}[Infimum of $\mathcal{F}_{s,p}$]  \label{prop:optimal-dilation}
		Let $1 < p < \infty$,  $0<s\,p < N$ and let $\chi  >0$. For every $M > 0$, we have 
		\begin{equation*}
			\inf_{\rho \in \mathcal{Y}_M} \mathcal{F}_{s,p}(\rho) = \begin{cases}
				-\infty, \quad &\mbox{ if } 1 < m < m_c,\\
				\\
				\nu_s, \quad &\mbox{ if } m = m_c, \\
				\\
				\mu_s, \quad &\mbox{ if } m > m_c,
			\end{cases}
		\end{equation*}
		for some $\mu_s = \mu\left(N, p, s, \chi, m, M\right) < 0$, where $\nu_s = \nu_s\left(N,p,s, \chi, M\right)$ is given by 
		\[
		\nu_s = 
		\begin{cases}
			0, \qquad &\mbox{ if } 0< M \le M_c,\\
			\\
			-\infty, \qquad &\mbox{ if } M > M_c,
		\end{cases}
		\]
		being $M_c$ the critical mass introduced in \eqref{massa-critica}. 
	\end{proposition}
	\begin{proof}
		As in the beginning of the proof of Proposition \ref{prop:dom-regime}, we take $\rho \in \mathcal{Y}_M$ and consider the function  $f_{\rho}(\lambda):=\mathcal{F}_{s,p}(\rho^\lambda)$, $\lambda>0,$ whose expression is given by \eqref{defi:functional-lambda} for every $m>1$.\KKK
		\vskip.2cm \noindent
		If $1 < m < m_c$, by sending $\lambda \nearrow \infty$ we infer that, for every $M>0,$
		\begin{equation*}
			\inf_{\rho \in \mathcal{Y}_M} \mathcal{F}_{s,p}(\rho) =  \lim_{\lambda \to \infty} f_\rho(\lambda) = -\infty.
		\end{equation*} 
		\vskip.2cm \noindent
		If $m > m_c$, from Proposition, \ref{prop:dom-regime} and by recalling \eqref{mcintro} and \eqref{esponente-interpol}, we get 
		\begin{equation} \label{grazie}
			\Lambda(\rho)^{\frac{m-m_c}{m-1}} = \frac{\|K_{s/2} \ast \rho\|^{p'}_{p'}}{\|\rho\|_m^{\frac{m}{m-1} \frac{p'}{p^*_s}} } = \frac{\|K_{s/2} \ast \rho\|^{p'}_{p'}}{\|\rho\|_m^{p'\,(1-\vartheta_0)} }\MMM=M^{p'\vartheta_0}\,\left(\frac{\|K_{s/2} \ast \rho\|^{p'}_{p'}}{M^{p'\vartheta_0}\,\|\rho\|_m^{p'\,(1-\vartheta_0)} }\right)\KKK. 	
		\end{equation}
		\MMM Since $m>m_c>(p_s^*)'$, \KKK the Hardy-Littlewood-Sobolev type inequality \eqref{ineq-hls-2} entails therefore that 
		\[
		\sup_{\rho \in \mathcal{Y}_M} \Lambda(\rho) \MMM\in(0,+\infty)\KKK
		\]
		for every $M>0$.
		In turn, by Proposition \ref{prop:dom-regime}   $$\inf_{\rho \in \mathcal{Y}_M}\mathcal F_{s,p}(\rho)=\inf_{\rho \in \mathcal{Y}_M}(\kappa\,\Lambda(\rho))=\kappa\sup_{\rho \in \mathcal{Y}_M} \Lambda(\rho)$$ since $\kappa<0$, \KKK thus
		\begin{equation*}
			\mu_s = \mu(N, p, s, \chi, m ,M) := \inf_{\rho \in \mathcal{Y}_M} \mathcal{F}_{s,p}(\rho) \in (-\infty, 0).
		\end{equation*}
		\vskip.2cm \noindent
		If $m = m_c$, we need to distinguish two cases. 
		\vskip.2cm
		{\it Case $0< M \le M_c$.} We take $\rho \in \mathcal{Y}_M$ and test the energy $\mathcal{F}_{s,p}$ with its mass invariant dilations $\rho^\lambda$ given by \eqref{def:dilation}.
		By the change of variable formula, we have  
		\[
		\|\rho^\lambda\|_{m_c}^{m_c} = \lambda^{N(m_c-1)}\|\rho\|_{m_c}^{m_c}, \qquad \mbox{ for }\,\, \lambda > 0.
		\]
		Since $m_c > 1$, using  Proposition \ref{prop:fair-comp} and \KKK   sending $\lambda \searrow 0$ we get
		\begin{equation} \label{quasi-critico}
			\inf_{\rho \in \mathcal{Y}_M} \mathcal{F}_{s,p}(\rho) = \lim_{\lambda \to 0} \mathcal{F}_{s,p}(\rho^\lambda) = 0 .
		\end{equation}
		\vskip.2cm 
		{\it Case $M > M_c.$} Let $h_s$ be the  unit mass \KKK extremal (with barycenter at the origin) of the HLS-type inequality \eqref{ineq-hls-2}, provided by Lemma \ref{lm:extramls-hls}. We set for every $\lambda>0$
		\[
		\rho_\lambda(x):= M\,\lambda^N\,h_s(\lambda\,x) \in \mathcal{Y}_M. 
		\] 
		By recalling \eqref{mcintro} and \eqref{massa-critica}, we have 
		\[
		\begin{split}
			\mathcal{F}_{s,p}(\rho_\lambda) &= \frac{\|\rho_\lambda\|_{m_c}^{m_c}}{m_c-1} - \frac{\chi}{p'}\,\|K_{s/2} \ast \rho_\lambda\|_{p'}^{p'} = \lambda^{2\,N\,m_c - N}\,\frac{M^{m_c}}{m_c-1} - \frac{\chi}{p'}\,\lambda^{(N-s)\,p' - N}\,M^{p'}\,H^*_{m_c,s} \\
			&= \lambda^{p'\,(N-s)}\,\frac{p^*_s}{p'}\,M^{p'} \left[\frac{1}{M^{\frac{s\,p'}{N}}} - \frac{1}{M_c^{\frac{s\,p'}{N}}}\right] < 0,
		\end{split}
		\]
		from which, by sending $\lambda \to \infty$, we get the desired conclusion. 
	\end{proof}
	\begin{remark} \label{rmk:non-existence}
		By the previous proposition, we infer that there are no minimizers in $\mathcal{Y}_M$ of the energy functional $\mathcal{F}_{s,p}$ in the aggregation dominated regime $m \in (1, m_c)$, whatever the value of the mass $M > 0$. 
		In the fair competition regime, $m = m_c$, still there are no minimizers in $\mathcal Y_M$ for prescribed mass $M \in (M_c, \infty).$ Also for values of the mass $M \in (0, M_c)$, there are no minimizers of $\mathcal{F}_{s,p}$ in $\mathcal{Y}_M$, in light of the leftmost inequality in Proposition \ref{prop:fair-comp}.  
	\end{remark}
	\begin{remark}\label{rmk:identita-potenziale-minimi} 
		By inspecting the proof of Proposition \ref{prop:dom-regime}  we can infer that, for every $m > 1$, any critical point  $\rho$ of $\mathcal{F}_{s,p}$  necessarily satisfies the relevant identity 
		\begin{equation} \label{identita-pot-minimi}
		\frac{\chi}{p^*_s} \|K_{s/2} \ast \rho\|^{p'}_{p'}  = \|\rho\|_{m}^{m}.
	\end{equation}
	Indeed, \eqref{defi:functional-lambda} holds for every $m>1$. Therefore imposing criticality of $\rho$ only with respect to mass invariant dilations, i.e. imposing \eqref{criticality},  yields \eqref{identita-pot-minimi}. Notice that if $1<m<m_c$ then $\rho$ is a maximum, and not a minumum, in the family $\{\rho^\lambda\}_{\lambda>0}$. Notice also that, in the case $m=m_c$, \eqref{identita-pot-minimi} is equivalent to $\mathcal F_{s,p}(\rho)=0$. 
	\end{remark}

		Next, we discuss the relation between extremals of \eqref{ineq-hls-2} and minimizers of $\mathcal{F}_{s,p}$. 
		
	\begin{corollary}[Extremals HLS vs minimizers of $\mathcal{F}_{s,p}$] \label{cor:invariant-functional}
		Let $1 < p < \infty$ and  $0<s\,p < N$. 
		
		Let $m \ge m_c$ and let $M > 0$. If $\rho_s \in \mathcal{Y}_M$ attains the infimum of the energy functional $\mathcal{F}_{s,p}$, among functions in $\mathcal{Y}_M$, then $\rho_s$ is an extremal of the Hardy-Littlewood-Sobolev type inequality \eqref{ineq-hls-2}, that is $\rho_{s}$ attains the supremum in \eqref{extremal-functional-hls}. 
		
		Viceversa, for $m > m_c$, if $M>0$ and $\rho_{s} \in \mathcal{Y}_M$ is an extremal of the Hardy-Littlewood-Sobolev type inequality \eqref{ineq-hls-2}, then its optimal dilation, in the sense of {\rm Proposition \ref{prop:dom-regime}}, attains the infimum of $\mathcal{F}_{s,p}$ over $\mathcal Y_M$. If $m = m_c$  and  $\rho_{s} \in \mathcal{Y}_{M_c}$ is an extremal of the Hardy-Littlewood-Sobolev type inequality \eqref{ineq-hls-2}, \KKK then $\rho_{s}$  is a minimizer of $\mathcal{F}_{s,p} $ over $\mathcal Y_{M_c}$. 
	\end{corollary}
	\begin{proof}
		Assume first that $m > m_c$, $M>0$. Let $\rho_s \in \mathcal{Y}_M$ be a minimizer of $\mathcal{F}_{s,p}$ over $\mathcal Y_M$. By Proposition \ref{prop:dom-regime}, we infer that $\rho_s$ maximizes  $\Lambda$ among functions in $\mathcal{Y}_M$,  where $\Lambda$ is the functional defined therein\KKK. In particular, by Remark \ref{rmk:invariant-by-dilations} \MMM and by \eqref{grazie}\KKK, $\rho_{s}$ is an extremal of the Hardy-Littlewood-Sobolev type inequality \eqref{ineq-hls-2}.
		On the other hand, let $\rho_{s} \in \mathcal{Y}_M$ be an extremal of the Hardy-Littlewood-Sobolev inequality \eqref{ineq-hls-2}
		(existence of an extremal in $\mathcal Y_M$ is guaranteed by Remark \ref{extremalym}).
		 Therefore it also satisfies
		\[
		\Lambda(\rho_s)\ge\Lambda(\rho)\qquad\mbox{for every $\rho\in\mathcal Y_M$},
		\]
		in view of \eqref{grazie} and Remark \ref{rmk:invariant-by-dilations}.
		\KKK Its optimal dilation in the sense of Proposition \ref{prop:dom-regime}, which is still an extremal of \eqref{ineq-hls-2} in light of Remark \ref{rmk:invariant-by-dilations}, is given by
		\[
		\widetilde{\rho_s} := \rho_{s}^{\lambda_{*}(\rho_{s})}, 
		\]
		for $\lambda_{*}(\rho_{s})$ as in  \eqref{fattore-dilatazione-rho00}. We claim that  $\widetilde{\rho_s}$	minimizes the energy functional $\mathcal{F}_{s,p}$ among all functions in $\mathcal{Y}_M$. Indeed, for every $\rho \in \mathcal{Y}_M$,  by using the   maximality \KKK of $\rho_{s}$,  the invariance by dilations property of $\Lambda$ and  Proposition \ref{prop:dom-regime}, we have 
		\[
		\mathcal F_{s,p}(\widetilde{\rho_s})=\mathcal F_{s,p}(\rho_s^{\lambda_*(\rho_s)})=\kappa\,\Lambda(\rho_s)\le \kappa\,\Lambda(\rho)\le\mathcal F_{s,p}(\rho)
		\]
		\KKK
		where the first inequality comes by recalling that $\kappa$ is negative. This proves the claim.
		
		 If $m = m_c$, take any minimizer $\rho_{s} \in \mathcal{Y}_{M_c}$ of $\mathcal{F}_{s,p}$ over $\mathcal Y_{M_c}$. By Proposition \ref{prop:optimal-dilation} we have $\mathcal{F}_{s,p}(\rho_{s}) = 0$, that is 
		\begin{equation} \label{cor-hls-vs-minimi1}
			\frac{1}{m_c - 1}\,\|\rho_{s}\|_{m_c}^{m_c} = \frac{\chi}{p'}\,\|K_{s/2} \ast \rho_s\|_{p'}^{p'}. 
		\end{equation}
		By recalling \eqref{massa-critica} and \eqref{extremal-functional-hls}, this entails that 
		\begin{equation} \label{cor-hls-vs-minimi}
			\frac{\|K_{s/2} \ast \rho_{s}\|_{p'}^{p'}}{M_c^{\frac{p'\,s}{N}}\,\|\rho_{s}\|_{m_c}^{m_c}} = H^*_{m_c, s},
		\end{equation}
		as desired. On the other hand, let $\rho_{s} \in \mathcal{Y}_{M_c}$ be an extremal of \eqref{ineq-hls-2}. Then, always by recalling \eqref{massa-critica}, \eqref{cor-hls-vs-minimi} implies \eqref{cor-hls-vs-minimi1} that is $\mathcal{F}_{s,p}(\rho_{s}) = 0$,  proving that $\rho_s$ minimizes $\mathcal F_{s,p}$ over $\mathcal Y_{M_c}$ in view of Proposition \ref{prop:optimal-dilation}. \KKK 
	\end{proof}

	The equation satisfied by the global minimizers of $\mathcal{F}_{s,p}$, provided by  Corollary \ref{cor:invariant-functional}\KKK, reads as follows. 
	
	\begin{lemma} \label{lm:euler-lagrange}
		Let $1 < p < \infty$ and  $0<s\,p < N$, and let $\chi, M > 0$. For $m > m_c$,  if $\rho_s \in \mathcal{Y}_M$ is a minimizer of $\mathcal{F}_{s,p}$ over $\mathcal{Y}_M$ then it solves 
		\begin{equation} \label{euler-lagrange-eqnfinale}
			\frac{m}{m-1}\rho_s^{m-1} = \Big(\chi\,\mathcal{K}_{s,p}(\rho_{s}) - \mathcal{D}_s\Big)_{+} \qquad \mbox{ in } \mathbb{R}^N,
		\end{equation}
		 with \begin{equation}\label{ds}\mathcal{D}_s: = \left(\frac{p^*_s - m'}{M} \right) \|\rho_{s}\|_m^m = \left(\frac{p^*_s - m'}{M} \right)  \frac{\chi}{p^*_s}\|K_{s/2} \ast \rho_{s}\|_{p'}^{p'}.\end{equation} \KKK
		In particular, it is bounded with compact support, radially nonincreasing, $W^{1,1}(\mathbb R^N)$ and  satisfies the H\"older regularity properties of {\rm Lemma \ref{lm:holder}}. The same holds for $m=m_c$, by taking $M = M_c$.\KKK 
	\end{lemma}
	\begin{proof}
		Let $\rho_{s} \in \mathcal{Y}_M$ be as in the statement. By Corollary \ref{cor:invariant-functional}, we have that $\rho_{s}$ is an extremal of the Hardy-Littlewood-Sobolev type inequality \eqref{ineq-hls-2}.  In light of Lemma \ref{lm:el-estremali-hls}, it satisfies \eqref{eqn_el}.		If $m>m_c$, Proposition \ref{prop:dom-regime} implies that $\rho_s$ coincides with its optimal dilation, i.e., ${\lambda_*(\rho_s)}=1$, where $\lambda_*$ is given by \eqref{fattore-dilatazione-rho00}, so that \eqref{identita-pot-minimi} holds. If $m=m_c$, then Remark \ref{rmk:non-existence} and Lemma \ref{prop:optimal-dilation} imply $M=M_c$ and $\mathcal F_{s,p}(\rho_s)=0$ which in turn directly yields \eqref{identita-pot-minimi}. 
	By inserting \eqref{identita-pot-minimi} in \eqref{costanti-eulero-lagrange}, we see that \eqref{eqn_el} becomes \eqref{euler-lagrange-eqnfinale}-\eqref{ds}. \KKK
	By Corollary \ref{W11}, Lemma \ref{lm:bootstrap-hls} and Lemma \ref{lm:holder} we then infer the desired conclusions.  	
	\end{proof}
	\begin{remark} \label{rmk:costante-el}
		For future purposes we record the following identities,  valid for $m>m_c$, involving the constant $\mathcal D_{s}$ appearing in the Euler-Lagrange equation \eqref{euler-lagrange-eqnfinale} satisfied by a minimizer of $\mathcal F_{s,p}$ over $\mathcal Y_M$:
		\begin{equation*} 
			\mathcal{D}_s = \left(\frac{p^*_s - m'}{M} \right) \|\rho_{s}\|_m^m = \left(\frac{p^*_s - m'}{M} \right)  \frac{\chi}{p^*_s}\|K_{s/2} \ast \rho_{s}\|_{p'}^{p'} = \left(\frac{p^*_s - m'}{M} \right) \frac{(m_c-1)(m-1)}{m_c - m} \mathcal{F}_{s,p}(\rho_{s}),
		\end{equation*}
		which follows from \eqref{ds} and the definition of $\mathcal F_{s,p}$.\KKK
	\end{remark}
	\begin{remark}\label{lastremark} 
	\MMM
		The conclusion of Lemma \ref{lm:euler-lagrange} holds also   for continuous radially decreasing critical points of $\mathcal F_{s,p} $ over $\mathcal Y_M$ in the case $(p^*_s)' < m < m_c$. These are defined as continuous radially decreasing solutions to \eqref{euler-lagrange-eqnfinale}, with $\mathcal D_s$ still expressed by \eqref{ds}.
	 Indeed, a first variation argument along the line of Lemma \ref{lm:el-estremali-hls} proves that the Euler-Lagrange equation that is necessarily satisfied by a continuous radially decreasing critical point $\rho$ of $\mathcal F_{s,p}$, constrained to $\mathcal Y_M$, is of the form
\eqref{ELintro2},
	  for some suitable constant $\mathcal Q$ having the role of Lagrange multiplier for the mass constraint. Moreover $\mathcal Q$ necessarily coincides with $\mathcal D_s$ from \eqref{ds} in view of the criticality condition \eqref{identita-pot-minimi}.
		Existence of such critical points, for every mass $M>0$, is deduced from the existence of radially decreasing extremals of the HLS inequality \eqref{ineq-hls-2} having mass $M$ and satisfying \eqref{identita-pot-minimi}, which is guaranteed by Lemma \ref{lm:extramls-hls} and Remark \ref{rmk:invariant-by-dilations}: notice indeed that an extremal of mass $M$ satisfies \eqref{identita-pot-minimi} after taking a mass invariant dilation (thus preserving extremality), see Remark \ref{rmk:identita-potenziale-minimi}. A HLS extremal having mass $M$ and satisfying \eqref{identita-pot-minimi} does  satisfy  \eqref{euler-lagrange-eqnfinale}-\eqref{ds}, thanks to Lemma \ref{lm:el-estremali-hls}, as seen  by plugging \eqref{identita-pot-minimi} in \eqref{eqn_el}-\eqref{costanti-eulero-lagrange}. The further regularity of such critical points is then deduced in the same way starting from the Euler-Lagrange equation, see Lemma \ref{lm:bootstrap-hls}, Lemma \ref{lm:holder} and Corollary \ref{W11}.\KKK
	\end{remark}

	\begin{proof}[\bf Proof of Theorem \ref{main1}] 
		The first part of Theorem \ref{main1} follows by Lemma \ref{lm:extramls-hls}, Lemma \ref{lm:el-estremali-hls}, Lemma \ref{lm:bootstrap-hls} and Lemma \ref{lm:holder}. The second part  follows by Corollary \ref{cor:invariant-functional} and Lemma \ref{lm:euler-lagrange}. \KKK
	\end{proof}
	\section{The limit $s\to0$} \label{sec:asymptotics}		
	This section is devoted to study the asymptotic behavior of the minimizers $\rho_s$, provided by  Corollary \ref{cor:invariant-functional}, \KKK as $s$ tends to $0$. Since the critical exponent $m_c$ given by \eqref{mcintro} tends to $p'$, as $s$ goes to zero, we need discuss separately three cases according to whether $m < p'$, $m = p'$ or $m > p'$. 
	
	If $m < p'$, we have  $m < (p_s^*)'$ \KKK for $s$ small enough. By Proposition \ref{prop:optimal-dilation}, we get $\inf_{\rho \in \mathcal{Y}_M} \mathcal{F}_{s,p} = -\infty$ thus there are no minimizers of $\mathcal{F}_{s,p}$,  and not even stationary states obtained from extremals of the HLS inequality according to Remark \ref{lastremark}. The case $m = p'$ (that we call the limiting fair competition regime) will be discussed in Section \ref{subsection:section-m=p'}.
	The limiting diffusion dominated  regime $m > p'$   is the most interesting one and it will be treated in the first part of this section. The main property is contained in Theorem \ref{thm:convergenza-forte-minimizzanti} below: we will prove that if $ m > p'$, any family of minimizers $(\rho_{s})_{s \in (0,N/p)}$ of the free energy functional $\mathcal{F}_{s,p}$ provided by  Corollary \ref{cor:invariant-functional} \KKK strongly converges as $s \searrow 0$ to the unique minimizer  in $\mathcal{Y}_M$ of a limit functional $\mathcal{F}_0$. In addition, $\mathcal{F}_{s,p}$ $\Gamma-$converges to $\mathcal{F}_0$ on $\mathcal{Y}_M$, with respect to the strong convergence in $L^{p'}(\mathbb{R}^N)$, see Proposition \ref{prop:gamma-convergenza}.  
	
	\subsection{The limit functional }
Concerning the limit functional $\mathcal F_0$ defined by \eqref{limit-functional}
	 we have the following
	\begin{proposition} \label{prop:gamma-limite}
		Let $1 < p < \infty$ and $m > p'$. There exists a unique radially symmetric and nonincreasing minimizer of $\mathcal{F}_0$ in $\mathcal{Y}_M$, given by
		\begin{equation} \label{minimo-gamma-limite}
			\rho_0(x) = \left(\frac{\chi}{p}\right)^{\frac{1}{m-p'}} 1_{B_{R_0}}(x), \qquad \mbox{ where }\,\, R_0 = \left(\frac{M}{\omega_N}\left(\frac{p}{\chi}\right)^{\frac{1}{m-p'}} \right)^{\frac{1}{N}}. 
		\end{equation}
	\end{proposition}
	\begin{proof}
		Let $\rho \in \mathcal{Y}_M$ and let $\rho_\lambda$ be its mass invariant dilation given by \eqref{def:dilation}. We have 
		\[
		\mathcal{F}_0(\rho^\lambda) = \frac{\lambda^{N(m-1)}}{m-1}  \|\rho\|^m_m -  \frac{\chi}{p'} \lambda^{N(p'-1)} \|\rho\|_{p'}^{p'},
		\]
		and 
		\[
		\begin{split}
			\frac{\rm d}{\rm d\lambda} \mathcal{F}_0(\rho^\lambda) = N\lambda^{N(m-1) -1}  \|\rho\|^m_m - \frac{\chi}{p} N \lambda^{N(p'-1)-1} \|\rho\|_{p'}^{p'} 
			= N\,\lambda^{N(p'-1)-1}\left[\lambda^{N(m-p')} \|\rho\|^m_m - \frac{\chi}{p} \|\rho\|_{p'}^{p'}\right],
		\end{split}
		\]
		for $\lambda > 0$.
		By optimizing in $\lambda$, we find that 
		\begin{equation} \label{lambda-minimo}
			\lambda_*(\rho) := \left(\frac{\chi}{p}\,\frac{\|\rho\|_{p'}^{p'}}{\|\rho\|_m^m}\right)^{\frac{1}{N(m-p')}},
		\end{equation}
		is the unique global minimum of $\lambda\mapsto \mathcal{F}_0(\rho^\lambda)$, for $\lambda > 0$. We then have 
		\[
		\mathcal{F}_0(\rho^{\lambda_{*}(\rho)}) = \kappa\,\Lambda(\rho),
		\] 
		where 
		\[
		\kappa = -\left(\frac{\chi}{p}\right)^{\frac{m-1}{m-p'}} \qquad \mbox{ and } \qquad \Lambda(\rho) = \frac{\|\rho\|_{p'}^{p'\,\frac{m-1}{m-p'}}}{\|\rho\|_m^{m\,\frac{p'-1}{m-p'}}}. 
		\]
		In order to minimize the functional $\mathcal{F}_0$ on $\mathcal{Y}_M$ we can equivalently maximize $\Lambda$.	Moreover, by symmetrization, we can look for maximizer of $\Lambda$ in the restricted class of $$\widetilde{\mathcal{Y}}_M :=  \left\{\rho \in \mathcal{Y}_M\,:\, \rho \mbox{ is radially symmetric and nonincreasing} \right\}.$$
		By using H\"older's inequality, we get that for $ \rho \in \widetilde{\mathcal{Y}}_M$ 
		\begin{equation*} \label{holder-minimi}
			\|\rho\|_{p'}^{p'} \le M^{\frac{m-p'}{m-1}}\,\|\rho\|_m^{m\,\frac{p'-1}{m-1}} \quad \Longleftrightarrow \quad \Lambda(\rho) \le M,
		\end{equation*}
		and the equality is satisfied by a function $\rho_{0}$ if and only if $\rho_0(x) = c\,1_F(x),$ for some measurable set $F \subseteq \mathbb{R}^N$, see \cite[Theorem 2.3 (ii.b)]{LL} for example. Since $\rho_0 \in \widetilde{\mathcal{Y}}_M$ and by using that $\|\rho_0\|_1 = M$, we infer that $F = B_R$, for some $R > 0$, and $c = M/|B_R|$. Moreover, since $\rho_{0}$ minimizes $\mathcal{F}_0$ it must coincide with its optimal dilation given by \eqref{lambda-minimo}, i.e. $\lambda_*(\rho_0) = 1$. This entails that 
		\[
		M = \Lambda(\rho_{0}) = \left(\frac{p}{\chi}\right)^{\frac{m-1}{m-p'}}\,\frac{M^m}{|B_R|^{m-1}},   
		\]
		from which we infer \eqref{minimo-gamma-limite}. 
	\end{proof}
	In the next result we infer information on the limiting behavior of the minimum value of $\mathcal{F}_{s,p}$, by testing the energy $\mathcal{F}_{s,p}$ with $\rho_{0}$.  
	\begin{corollary} \label{staccati-da-zero}
		Let $1 < p < \infty$ and   $0<s\,p < N$. Let $m > p'$ and let $\chi, M > 0$. If $\rho_s \in \mathcal{Y}_M$ is a minimizer of $\mathcal{F}_{s,p}$ over $\mathcal{Y}_M$ \MMM for every $s\in(0,N/p)$\KKK, then we have 
		\[
		\limsup_{s \to 0} \mathcal{F}_{s,p}(\rho_{s}) < 0\quad\mbox{and}
\quad \limsup_{s \to 0} \mathcal{D}_s > 0		\]
		 where $\mathcal{D}_s > 0$ is the constant related to $\rho_s$ appearing in {\rm Lemma \ref{lm:euler-lagrange}}. \KKK
	\end{corollary}
	\begin{proof}
		Let $\rho_{0}$ be as in Proposition \ref{prop:gamma-limite}. By using Theorem \ref{thm:kurokawa}, we have 
		\begin{equation*} \label{claim-asym}
			\lim_{s \to 0} \|K_{s/2} \ast \rho_{0} - \rho_{0}\|_{p'} = 0. 
		\end{equation*}
	 By the minimality of $\rho_{s}$ and by using the explicit expression of $\rho_{0}$, this entails that   
		\[
		\begin{split}
			\limsup_{s \to 0} \mathcal{F}_{s,p}(\rho_{s}) \le \lim_{s \to 0} \mathcal{F}_{s,p}(\rho_{0})
			= \frac{M}{m-1} \left(\frac{\chi}{p}\right)^{\frac{m-1}{m-p'}} - M \frac{\chi}{p'} \left(\frac{\chi}{p}\right)^{\frac{p'-1}{m-p'}} < 0,
		\end{split}
		\]
		where the last inequality follows since $m > p'$. By recalling Remark \ref{rmk:costante-el}, we also get the announced asymptotic behavior of $\mathcal{D}_s$. 
	\end{proof}
	
	\MMM
	
	\begin{remark}\label{limitfair} Concerning the limit functional $\mathcal F_0$ in the case $m=p'$, for any $M>0$
	it is clear that $\inf_{\mathcal Y_M}\mathcal F_0=0$ if $0<\chi\le p$ and that  $\inf_{\mathcal Y_M}\mathcal F_0=-\infty$ if $\chi> p$. These properties are obtained by taking dilations $\rho^\lambda$ for any given $\rho\in\mathcal Y_M$ and by sending $\lambda$ to $+\infty$ and to $0$, respectively. The infimum is not realized, except for the trivial case $p=\chi$. 
	\end{remark}
	
	\KKK

	\subsection{The limiting diffusion dominated regime}
	Next, we discuss the limiting behavior of the minimizers for $s \searrow 0$, in the case $m>p'$.
	\begin{proposition}[Equiboundedness of $\rho_{s}$] \label{equibounded}
		Let $1 < p < \infty$ and  $0<s\,p < N$. Let $m > p'$ and let $\chi, M > 0$. Let $\rho_s \in \mathcal{Y}_M$ be a minimizer of $\mathcal{F}_{s,p}$ over $\mathcal{Y}_M$ \MMM for every $s\in(0,N/p)$\KKK. \MMM Then there exists $s_0\in(0,N/p)$ such that \KKK
		\[
		\sup_{s \in (0, s_0)}\|\rho_{s}\|_{\infty} < \infty \qquad \mbox{ and } \qquad \sup_{s \in (0, s_0)} \|\rho_{s}\|_m < \infty.
		\]
	\end{proposition}
	\begin{proof}
		From Corollary \ref{cor:invariant-functional} and Lemma \ref{lm:euler-lagrange} we know that $\rho_{s}$ is a radially symmetric nonincreasing H\"older continuous function. This yields that 
		\begin{equation} \label{eqn:stime-puntuali}
			\|\rho_s\|_{\infty} = \rho_s(0) \qquad \mbox{ and } \qquad \rho_s(x) \le \frac{M}{\omega_N |x|^N} \qquad \mbox{ for every } x \in \mathbb{R}^N \setminus \{0\},
		\end{equation}
		for every $s$. The first fact is clear, and the second one follows since we have 
		\[
		M \ge \int_{B_{|x|}} \rho_{s}\, dy \ge \int_{B_{|x|}} \rho(x)\, dy = \rho_{s}(x)\, \omega_N\, |x|^N, \qquad \mbox{ for every } x \in \mathbb{R}^N \setminus \{0\}.
		\] 
		By using the Euler-Lagrange equation \eqref{eqn_el} and H\"older's inequality we get
		\[
		\begin{split}
			\frac{m'}{\chi} \rho_{s}(0)^{m-1} &\le  K_{s/2} \ast (K_{s/2} \ast \rho_{s})^{p'-1}(0) \\
			&= c_{N, s/2}\,\int_{B_1} |y|^{s-N} \left(K_{s/2} \ast \rho_s\right)^{p'-1}(y)dy + c_{N, s/2}\int_{B^c_1} |y|^{s-N} \left(K_{s/2} \ast \rho_s\right)^{p'-1}(y)dy\\
			&\le c_{N, s/2}\,\left(\frac{N\omega_N}{s}\right)\,\|K_{s/2} \ast \rho_{s}\|^{p'-1}_\infty +  c_{N, s/2}\,\frac{N \omega_N}{(N-s)\,p' + N}\,\|K_{s/2} \ast \rho_{s}\|_{p'}^{\frac{p'}{p}}\\
			&\le c_{N, s/2}\,\left(\frac{N\omega_N}{s}\right)\,\left(\alpha_s\,M + \beta_s\,\rho_{s}(0)\right)^{p'-1}  +  c_{N, s/2}\,\frac{N \omega_N}{(N-s)\,p' + N}\,H_s^{p'-1}\,\|\rho_{s}\|^{p'-1}_{(p^*_s)'},
		\end{split}
		\]
		where in the last line we used Lemma \ref{lm:bound-l-infinito-riesz} with data $q:=1$ and $r:= \infty$ and the HLS inequality \eqref{ineq:hls-particolare}. By spending again \eqref{eqn:stime-puntuali}, we infer 
		\[
		\int_{\mathbb{R}^N}  \rho_{s}^{(p^*_s)'}\,dx = \int_{B_{1}} \rho_{s}^{(p^*_s)'}\,dx + \int_{B^c_1} \rho_{s}^{(p^*_s)'}\,dx \le \omega_N\,\rho_{s}(0)^{(p^*_s)'} + \left(\frac{M}{\omega_N}\right)^{(p^*_s)'} \,\frac{1}{N\left((p^*_s)'-1\right)} 
		\]
		By collecting the last two inequalities, we get
		\[
		\begin{split}
			\frac{m'}{\chi} \rho_{s}(0)^{m-1} \le& c_{N, s/2}\,\left(\frac{N\,\omega_N}{s}\right)\,\left(\alpha_s\,M + \beta_s\,\rho_{s}(0)\right)^{p'-1}  \\ &+  c_{N, s/2}\frac{N\,\omega_N\,H_s^{p'-1}}{(N-s)\,p' + N}\left(\rho_{s}(0)^{(p^*_s)'}\omega_N + \left(\frac{M}{\omega_N}\right)^{(p^*_s)'} \frac{1}{N\,\left((p^*_s)'-1\right)}\right)^{\frac{p'-1}{(p^*_s)'}}. 
		\end{split}
		\]
		By contradiction, we assume now that 
		\[
		\limsup_{s \to 0} \rho_{s}(0) = +\infty,
		\]
		and we divide both sides of the previous inequality by $\rho_{s}(0)^{p'-1}$. By sending $s \searrow 0$, since $m > p'$ and by recalling the asymptotic behaviors of $\alpha_s, \beta_s, H_s$ and $c_{N, s/2}$ given respectively in Lemma \ref{lm:bound-l-infinito-riesz}, Corollary \ref{cor:hls} and \eqref{eqn:asym-behav}, we obtain a contradiction. This proves that there exists $s_0 >0$ such that 
		$
		 S:=\sup_{s \in (0, s_0)} \rho_{s}(0) < \infty. 
		$
		We conclude the proof by observing that we can also infer the equiboundendness of $\|\rho_{s}\|_m$ for $s \in (0, s_0)$, since we have 
		$
		\|\rho_{s}\|_m \le S^{\frac{1}{m'}}\,M^{\frac{1}{m}}, \mbox{ for } s \in (0, s_0),
		$     
		in light of the interpolation inequality in $L^p$ spaces. 
	\end{proof}
	\begin{proposition}[Equiboundedness of ${\rm supp}(\rho_{s})$] \label{supports-equibounded}
		Let $1 < p < \infty$ and $0<s\,p < N$. Let $m > p'$ and let $\chi, M > 0$. Let $\rho_s \in \mathcal{Y}_M$ be a minimizer of $\mathcal{F}_{s,p}$ over $\mathcal{Y}_M$ \MMM for every $s\in(0,N/p)$\KKK. Then there exist $ s_0 \in(0,N/p)$ and $R_0 > 0$ such that ${\rm supp}(\rho_{s}) \subseteq B_{R_0}$ for every $0 < s < s_0$.     
	\end{proposition}
	\begin{proof}
		\MMM Let $s_0$ given by  Proposition \ref{equibounded}. We set $B_{R_s} = {\rm supp}(\rho_{s})$ and by contradiction, we assume that 
		\begin{equation*} \label{hp-assurdo}
			\limsup_{s \to 0} R_s = +\infty.
		\end{equation*}
		This entails that $R_s > 1$, for small value of $s$.
		We preliminary observe that, from Corollary \ref{cor:hls} and Proposition \ref{equibounded}, we have 
		\begin{equation} \label{equilimitatezza-norma-pot}
			\|K_{s/2} \ast \rho_{s}\|_{p'} \le H_s\,\|\rho_{s}\|_{(p^*_s)'} \le H_s\,M^\vartheta\,\|\rho_{s}\|^{1-\vartheta}_\infty \le L, \qquad \mbox{ for } s \in (0, s_0),
		\end{equation}
		where $\vartheta = 1/(p^*_s)'$ and $L > 0$ is a constant dending only on  $s_0$. We take $x \in \partial B_{R_s}$, and by using Lemma \ref{lm:euler-lagrange} we get \KKK
		\begin{equation} \label{quasi}
			\begin{split}
				\frac{1}{\chi}\,\mathcal{D}_s &= K_{s/2} \ast \left(K_{s/2} \ast \rho_{s} \right)^{p' - 1}(x)
				= c_{N, s/2}\,\int_{B_1} | y|^{s-N} \left(\left(K_{s/2} \ast \rho_{s}\right)(x-y)\right)^{p'-1}\,dy \\& + c_{N, s/2}\,\int_{B^c_1} |y|^{s-N} \left(\left(K_{s/2} \ast \rho_{s}\right)(x-y)\right)^{p'-1}\,dy \\
				&=: \mathcal{A}_1 + \mathcal{A}_2.
			\end{split}
		\end{equation}
		We estimate the last two integrals separately, starting from the second one. By using H\"older's inequality (we observe that,  by assumption $p' > N/(N-s)$, so we have  $(s-N)\,p' + N < 0$) we have 
		\[
		\begin{split}
			\mathcal{A}_2
			&\le c_{N,s/2}\,\left(\int_{B_{1}^c} |y|^{(s-N)p'}\,dy \right)^{\frac{1}{p'}} \,\left(\int_{B_{1}^c} \left((K_{s/2} \ast \rho_{s})(x-y)\right)^{p'}\,dy\right)^{\frac{1}{p}} \\
			&\le c_{N,s/2} \,\frac{N\omega_N}{(N-s)\,p' + N} \,\|K_{s/2} \ast \rho_{s}\|_{p'}^{\frac{p'}{p}}.
		\end{split}
		\]
		So by using \eqref{eqn:asym-behav} and \eqref{equilimitatezza-norma-pot} we get 
		\begin{equation} \label{limiteA2}
			\lim_{s \to 0}\mathcal{A}_2 = 0.
		\end{equation}
		We now consider the first integral 
		\[
		\begin{split}
			\mathcal{A}_1 =  c_{N, s/2} \int_{B_1} | y|^{s-N} \left(\left(K_{s/2} \ast \rho_{s}\right)(x-y)\right)^{p'-1}\,dy  
		\end{split}
		\]
		Since $\rho_{s}$ is nonincreasing, so is $K_{s/2} \ast \rho_{s}$ \MMM (see  \cite{Can})\KKK. Then, by using also \eqref{equilimitatezza-norma-pot}, we get 
		\begin{equation*} \label{stima-potenziale}
			L \ge \left(\int_{\mathbb{R}^N} |K_{s/2} \ast \rho_{s}|^{p'}\,dy \right)^{\frac{1}{p'}} \ge \left(\int_{B_{|x|}} |K_{s/2} \ast \rho_{s}|^{p'}\,dy\right)^{\frac{1}{p'}} \ge \left(\omega_N |x|^N \right)^{\frac{1}{p'}}\,\left(K_{s/2} \ast \rho_{s}\right)(x),
		\end{equation*} 
		for every $x \in \mathbb{R}^N \setminus \{0\}$. 
		Since $|x| = R_s > 1$, by the triangle inequality we have 
		\[
		|x-y| \ge |x| - |y| \ge R_s - 1, \qquad \mbox{ for } y \in B_1,
		\] 
		thus
		\[
		\begin{split}
			\mathcal{A}_1 \le L\,\omega_N^{-\frac{1}{p}}\,c_{N, s/2}\,\int_{B_1} |y|^{s-N} |x-y|^{-\frac{N}{p}}\,dy &\le \frac{L\,\omega_N^{-\frac{1}{p}}}{\left(R_s - 1\right)^{\frac{N}{p}}}\,c_{N, s/2}\,\int_{B_{1}} |y|^{s-N}\,dy \\
			&= \frac{L\,\omega_N^{-\frac{1}{p}}}{\left(R_s - 1\right)^{\frac{N}{p}}}\,c_{N, s/2}\,\left(\frac{N\omega_N}{s}\right). 
		\end{split}
		\] 
		By recalling \eqref{eqn:asym-behav} and \eqref{hp-assurdo}, we eventually get that 
		$
			\lim_{s \to 0} \mathcal{A}_1 = 0,
		$
		and so also 
		$
		\limsup_{s \to 0} \mathcal{D}_s = 0,
		$
		by \eqref{quasi} and \eqref{limiteA2}. This contradicts Corollary \ref{staccati-da-zero} and gives the desired result. 
	\end{proof}
	\begin{proposition} \label{prop:compattezza}
		Let $1 < p < \infty$ and $0<s\,p < N$. Let $m > p'$, let $\chi, M > 0$ and let $\rho_s \in \mathcal{Y}_M$ be a minimizer of $\mathcal{F}_{s,p}$ over $\mathcal{Y}_M$ \MMM for every $s\in(0,N/p)$\KKK.  There exists $s_0 \in (0, N/p)$ such that the family of minimizers $( \rho_s )_{s \in (0, s_0)}$ is equibounded in \MMM $W^{1,1}(\mathbb R^N)$. Moreover, if $(s_k)_{k \in \mathbb{N}} \subseteq (0, s_0)$ converges to $0$, then the family $(\rho_{s_k})_{k\in\mathbb N}$ admits limit points in the strong $L^1(\mathbb R^N)$ topology, and if $\rho$ is a limit point along a not relabeled subsequence we have $\rho\in\mathcal Y_M\cap L^\infty(\mathbb R^N)$ and
		\[
		\lim_{k \to \infty} \|\rho_{s_k} - \rho\|_q = 0, \qquad \mbox{ for every } q \in [1, \infty). \KKK
		\]
	\end{proposition}
	\begin{proof}
		The proof is the same of \cite[Lemma 3.7]{HMVV} by using Proposition \ref{equibounded} and Proposition \ref{supports-equibounded}. 
	\end{proof}
	The main result of this section regarding the case $m>p'$ reads as follows
	\begin{theorem} \label{thm:convergenza-forte-minimizzanti}
		Let $1 < p < \infty$ and  $0<s\,p < N$. Let $m > p'$, let $\chi, M > 0$ and let $\rho_s \in \mathcal{Y}_M$ be a minimizer of $\mathcal{F}_{s,p}$ over $\mathcal{Y}_M$ \MMM for every $s\in(0,N/p)$\KKK. Then 
		\[
		\lim_{s \to 0} \|\rho_{s} - \rho_{0}\|_q = 0, \qquad \mbox{ for every } q \in [1, \infty),
		\]
		where $\rho_{0}$ is given by \eqref{minimo-gamma-limite}.
	\end{theorem}
	\begin{proof}
		By using Proposition \ref{prop:compattezza}, there exists a sequence $(s_k)_{k \in \mathbb{N}}$ converging to $0$ and a function $\rho \in \mathcal{Y}_M \cap L^\infty(\mathbb{R}^N)$ such that  
		\begin{equation} \label{hp-convergenza-forte}
			\lim_{k \to \infty} \|\rho_{s_k} - \rho\|_m = 0. 
		\end{equation}
		By the triangle inequality and the HLS-type inequality \eqref{ineq-hls-2}, we get 
		\[
		\begin{split}
			\|K_{s_k/2} \ast \rho _{s_k} - \rho\|_{p'} &\le \|K_{s_k/2} \ast \left(\rho_{s_k} - \rho\right)\|_{p'} + \|K_{s_k/2} \ast \rho - \rho\|_{p'} \\
			&\le  H_s\,(2\,M)^{\vartheta_0}\, \|\rho_{s_k} - \rho\|_m^{1-\vartheta_0} + \|K_{s_k/2} \ast \rho - \rho\|_{p'}. 
		\end{split}
		\]
		where $\vartheta_0$ is given by \eqref{esponente-interpol}, \MMM it depends on $s_k$ and converges to $1-m'/p>0$ as $k\to+\infty$. By Theorem \ref{thm:kurokawa}, we have 
		\[
		\lim_{k \to \infty} \|K_{s_k/2} \ast \rho - \rho\|_{p'} = 0,
		\]
		thus from the previous inequality, \MMM from  \eqref{hp-convergenza-forte} and from the bound on $H_s$ by Corollay \ref{cor:hls}, we get 
		\[
		\lim_{k \to \infty} \|K_{s_k/2} \ast \rho_{s_k} - \rho\|_{p'} = 0. 
		\]
		This entails that 
		\[
		\lim_{k \to \infty} \mathcal{F}_{s_k,p}(\rho_{s_k}) = \frac{\|\rho\|_m^m}{m-1} - \frac{\chi}{p'} \|\rho\|_{p'} = \mathcal{F}_0(\rho). 
		\]
		For every $\widetilde{\rho} \in \mathcal{Y}_M$, by using the previous equality, the minimality of $\rho_{s_k}$ and Theorem \ref{thm:kurokawa}, we infer 
		\[
		\mathcal{F}_0(\rho) = \lim_{k \to \infty} \mathcal{F}_{s_k,p}(\rho_{s_k}) \le \lim_{k \to \infty} \mathcal{F}_{s_k,p}(\widetilde{\rho}) = \mathcal{F}_0(\widetilde{\rho}).
		\] 
	 This proves that $\rho$ must be a minimizer of $\mathcal{F}_0$ in $\mathcal{Y}_M$ and so by Proposition \ref{prop:gamma-limite} it must coincide with $\rho_{0}$. By the arbitrariness of $(s_k)_{k \in \mathbb{N}}$, we eventually get that the whole family $(\rho_{s})$ strongly converges to $\rho_{0}$ in $L^q(\mathbb{R}^N)$, for every $q \in [1, \infty)$. 
	\end{proof}
	We further have the following $\Gamma-$convergence result. 
	\begin{proposition} \label{prop:gamma-convergenza}
	Let $1 < p < \infty$ and $\chi, M > 0$. Let $m > p'$. For $s \to 0$, the functional $\mathcal{F}_{s,p}$ $\Gamma-$converges to $\mathcal{F}_0$ on $\mathcal{Y}_M$ with respect to the strong convergence in $L^{p'}(\mathbb{R}^N)$.  
	\end{proposition}
	\begin{proof}
Let \MMM $(\rho_{s})_{s \in (0, N/p)} \subset \mathcal{Y}_M$ \KKK and $\rho \in \mathcal{Y}_M$ be such that 
		\begin{equation} \label{hp:gamma-conv}
			\lim_{s \to 0} \|\rho_{s} - \rho\|_{p'} = 0.
		\end{equation}
		By the triangle inequality and the Hardy-Littlewood-Sobolev \MMM inequality \MMM \eqref{ineq:hls-particolare} \KKK we have 
		\[
		\begin{split}
			\|K_{s/2} \ast \rho_{s} - \rho\|_{p'} &\le \|K_{s/2} \ast (\rho_{s} - \rho)\|_{p'} + \|K_{s/2} \ast \rho - \rho\|_{p'} \\
			& \le H_s\,\|\rho_{s} - \rho\|_{(p^*_s)'} + \|K_{s/2} \ast \rho - \rho\|_{p'} \\
			& \le H_s\,\|\rho_{s} - \rho\|^{\vartheta}_{1}\,\|\rho_{s} - \rho\|^{1-\vartheta}_{p'} + \|K_{s/2} \ast \rho - \rho\|_{p'},
		\end{split}
		\]
		where $\vartheta = 1 - p/p^*_s$. By Theorem \ref{thm:kurokawa}, we have
		\begin{equation*} \label{claim-duecasi}
			\lim_{s \to 0} \|K_{s/2} \ast \rho - \rho\|_{p'} = 0,
		\end{equation*}
		wich entails that 
		\begin{equation*} \label{quasi-fatto}
			\lim_{s \to 0} \|K_{s/2} \ast \rho_{s} - \rho\|_{p'} = 0,
		\end{equation*}
		where we also used our assumption \eqref{hp:gamma-conv}. By Fatou's lemma and by spending the last information, we infer   
		\begin{equation*}
			\mathcal{F}_0(\rho) \le \liminf_{s \to 0} \mathcal{F}_{s,p}(\rho_{s}). 
		\end{equation*}
		On the other hand, let $\rho \in \mathcal{Y}_M$ we set $\rho_{s} := \rho$, for every $s$. By using again Theorem \ref{thm:kurokawa} we have 
		\[
		\limsup_{s \to 0} \mathcal{F}_{s,p}(\rho_s) = \mathcal{F}_0(\rho),
		\]
		From the last two facts we obtain the claimed result. 
	\end{proof}
	\subsection{The limiting fair competition regime} \label{subsection:section-m=p'}
	We now analyze the limiting behavior of the minimizers in the remaining case $m = p'$ thus concluding the proof of Theorem \ref{main2}. We treat separately the cases $\chi \neq p$ and $\chi = p$, staring from the first one. 
	\begin{theorem} \label{thm:m=p'}
		Let $1 < p < \infty$ and  $0<s\,p < N$. Let $M > 0$ and $m = p'$. If $\rho_s \in \mathcal{Y}_M$ is a minimizer of $\mathcal{F}_{s,p}$ over $\mathcal{Y}_M$ \MMM for every $s\in(0,N/p)$\KKK, then we have 
		\begin{equation*}		
			\lim_{s \to 0} \|\rho_s\|_{\infty} = \MMM-\lim_{s\to 0}\mathcal F_{s,p}(\rho_s)\KKK =
			\begin{cases}
				0, \qquad &\mbox{ if } 0 < \chi < p,\\
				\\
				+\infty, \qquad &\mbox{ if } \chi > p.
			\end{cases} 
		\end{equation*}
		\MMM Moreover, if $\chi>p$ we have $\rho_s\to M\delta_0$ in the sense of measures as $s\to0$.\KKK
	\end{theorem}
	\begin{proof}
		By Corollary \ref{cor:invariant-functional}, we have that $\rho_{s}$ is radially symmetric and nonincreasing. We discuss separately two cases.
		\vskip.2cm \noindent 
		If $\boxed{0 < \chi < p}$, we argue by contradiction and assume that 
		\[\MMM
		\limsup_{s \to 0} \rho_{s}(0) > 0.\KKK
		\]
		By arguing as in the proof of Proposition \ref{equibounded}, we infer that 
		\[
		\begin{split}
			\frac{p}{\chi}\,\rho_{s}(0)^{p'-1} \le &\; c_{N, s/2}\,\left(\frac{N\omega_N}{s}\right)\,\left(\alpha_s\,M + \beta_s\,\rho_{s}(0)\right)^{p'-1}  \\ &+  c_{N, s/2}\,\frac{N\,\omega_N\,H_s^{p'-1}}{(N-s)\,p' + N}\left(\rho_{s}(0)^{(p^*_s)'}\,\omega_N + \left(\frac{M}{\omega_N}\right)^{(p^*_s)'} \frac{1}{N\,\left((p^*_s)'-1\right)}\right)^{\frac{p'-1}{(p^*_s)'}}. 
		\end{split}
		\]
		By dividing both sides of the previous inequality by $\rho_{s}(0)^{p'-1}$, we can rewrite it as
		\[
		\begin{split}
			&\frac{p}{\chi} - c_{N, s/2}\left(\frac{N\omega_N}{s}\right)\,\left( \frac{\alpha_s\,M}{\rho_{s}(0)} + \beta_s\right)^{p'-1} \\ &\quad\le  c_{N, s/2}\frac{N \omega_N\,H_s^{p'-1}}{(N-s)\,p' + N}\,\left(\omega_N + \left(\frac{M}{\omega_N\,\rho_{s}(0)}\right)^{(p^*_s)'}\,\frac{1}{N\,((p^*_s)' -1)}\right)^{\frac{p'-1}{(p^*_s)'}} = o(s).
		\end{split}
		\]
		By sending $s \to 0$ and by using \eqref{eqn:asym-behav}, \eqref{asym-bound-l-infinito-riesz1} and \eqref{asym-bound-l-infinito-riesz2}, we get a contradiction. \MMM Therefore $\rho_s$ converges uniformly to zero as $s\to 0$, which also implies $\mathcal F_{s,p}(\rho_s)\to0$, by Corollary \ref{cor:hls}.\KKK
		\vskip.2cm \noindent
		If $\boxed{\chi > p}$, we consider the limit functional $\mathcal{F}_0$, given by \eqref{limit-functional} with $m = p'$. For $\rho \in \mathcal{Y}_M$, we have
		\[
		\lim_{\lambda \to \infty} \mathcal{F}_0(\rho^\lambda) = 	\lim_{\lambda \to \infty} \lambda^{N\,(p'-1)}\,\left(\frac{1}{p'-1} - \frac{\chi}{p'}\right)\,\|\rho\|_{p'}^{p'} = -\infty,
		\] 
		where $\rho^\lambda \in \mathcal{Y}_M$ is the mass invariant dilation of $\rho$ by factor $\lambda$, given by \eqref{def:dilation}. This entails that 
		\[
		\inf_{\rho \in \mathcal{Y}_M} \mathcal{F}_0 = -\infty. 
		\] 
		Then if $\beta < 0$, we can take $\overline{\rho} \in \mathcal{Y}_M$ such that 
		$
		\mathcal{F}_0(\overline{\rho}) < \beta. 
		$
		By using that 
		\[
		\lim_{s \to  0} \|K_{s/2} \ast \overline{\rho} - \overline{\rho}\|_{p'} = 0,
		\]
		thanks to Theorem \ref{thm:kurokawa}, and by the minimality of $\rho_{s}$ we then obtain
		\[
		\limsup_{s \to 0} \mathcal{F}_{s,p}(\rho_{s}) \le \limsup_{s \to 0} \mathcal{F}_0(\overline{\rho}) = \mathcal{F}_0(\overline{\rho}) < \beta. 
		\]
		By the arbitrariness of $\beta$ and by using Remark \ref{rmk:costante-el}, this yields 
		\begin{equation} \label{asintotico-m=p'}
			\lim_{s \to 0} \mathcal{F}_{s,p}(\rho_{s}) = -\infty \quad \mbox{ and } \quad \lim_{s \to 0} \mathcal{D}_s = +\infty.
		\end{equation}
	 By contradiction, we assume that 
		\[
		R_0 := \limsup_{s \to 0} R_s > 0. 
		\]
		We take a sequence $(s_k)_{k \in \mathbb{N}} \subseteq (0,1)$ converging to zero and such that \begin{equation} \label{cazzatella}
			\lim_{k \to \infty} R_{s_k} = \limsup_{s \to 0} R_s = R_0.
		\end{equation}
		We set $\overline{R} = R_0/2$ and $ B_{R_{s_k}}:= {\rm supp}(\rho_{s_k})$. For $x \in \partial B_{R_{s_k}}$ by using Lemma \ref{lm:euler-lagrange} we have 
		\begin{equation} \label{quasi-asin-m=p'}
			\begin{split}
				\frac{1}{\chi}\,\mathcal{D}_{s_k} &= K_{{s_k}/2} \ast \left(K_{{s_k}/2} \ast \rho_{{s_k}} \right)^{p' - 1}(x) \\
				&= c_{N, {s_k}/2} \int_{B_{\overline{R}}} | y|^{{s_k}-N} \left(\left(K_{{s_k}/2} \ast \rho_{{s_k}}\right)(x-y)\right)^{p'-1}dy \\&\qquad + c_{N, {s_k}/2} \int_{B^c_{\overline{R}}} |y|^{{s_k}-N} \left(\left(K_{{s_k}/2} \ast \rho_{{s_k}}\right)(x-y)\right)^{p'-1}dy =: \mathcal{A}_1 + \mathcal{A}_2.
			\end{split}
		\end{equation}
		For $\mathcal{A}_1$, by arguing as in Proposition \ref{supports-equibounded}, we have that 
		\[
		\left(\omega_N |x|^N \right)^{\frac{1}{p'}}\,\left(K_{{s_k}/2} \ast \rho_{{s_k}}\right)(x) \le L, \qquad  \mbox{ for every } x \in \mathbb{R}^N \setminus \{0\},
		\]
		and, since $|x| = R_{s_k} > \overline{R}$ for $k$ large enough, by the triangle inequality we have 
		\[
		|x-y| \ge |x| - |y| \ge R_{s_k} - \overline{R}, \qquad \mbox{ for } y \in B_{\overline{R}},
		\] 
		This entails that 
		\begin{equation} \label{limite-1}
			\mathcal{A}_1 \le c_{N, {s_k}/2}\,\left(\frac{N\omega_N}{{s_k}}\right)\,\frac{L\,\omega_N^{-\frac{1}{p}}}{\left(R_{s_k} - \overline{R}\right)^{\frac{N}{p}}}.
		\end{equation}
		For $\mathcal{A}_2$, by using H\"older's inequality we have 
		\begin{equation} \label{limite-2}
			\begin{split}
				\mathcal{A}_2
				&\le c_{N,{s_k}/2}\, \left(\int_{B_{1}^c} |y|^{({s_k}-N)\,p'}\,dy \right)^{\frac{1}{p'}} \left(\int_{B_{1}^c} \left((K_{{s_k}/2} \ast \rho_{{s_k}})(x-y)\right)^{p'}\,dy\right)^{\frac{1}{p}} \\
				&\le c_{N,{s_k}/2} \,\frac{N\,\omega_N}{(N-{s_k})\,p' + N}\, \|K_{{s_k}/2} \ast \rho_{{s_k}}\|_{p'}^{\frac{p'}{p}} \\
				&\le E\,c_{N,{s_k}/2}\,\left(\frac{\mathcal{D}_{s_k}}{{s_k}}\right)^{\frac{1}{p}}, 
			\end{split}
		\end{equation}
		where in the last equality we used Remark \ref{rmk:costante-el} and we set 
		\[
		E = E\left(N,p, \chi, M\right):= \frac{N\omega_N}{(N-1)\,p' + N}\,\left(\frac{M}{\chi}\,\frac{N}{p}\right)^{\frac{1}{p}}. 
		\]
		By spending \eqref{limite-1} and \eqref{limite-2} in \eqref{quasi-asin-m=p'} and by dividing for $\mathcal{D}_{s_k}^{\frac{1}{p}}$, we get 
		\[
		\frac{1}{\chi}\,\mathcal{D}_{s_k}^{1- \frac{1}{p}} \le c_{N, {s_k}/2}\,\left(\frac{N\omega_N}{{s_k}}\right)\,\frac{L\,\omega_N^{-\frac{1}{p}}}{\left(R_{s_k} - \overline{R}\right)^{\frac{N}{p}}} + E\,c_{N,{s_k}/2}\,\left(\frac{1}{{s_k}}\right)^{\frac{1}{p}}. 
		\]
		In light of \eqref{eqn:asym-behav}, \eqref{asintotico-m=p'} and \eqref{cazzatella}, by sending $k \to \infty$, the last inequality yields a contradiction. This implies that 
		\[
		\lim_{s \to 0} R_s = 0, 
		\]
		thus, since $\rho_{s} \in \mathcal{Y}_M$, we must have that $\rho_s$ converges to a point mass at the origin as $s\to0,$
		as desired. \MMM Along with \eqref{asintotico-m=p'}, this concludes the proof. \KKK
	\end{proof}
	\begin{proposition}[Case $\chi = p$] \label{lastprop}
	Let $1 < p < \infty$ and  $0<s\,p < N$. Let $M > 0$, $m = p'$ and $\chi = p$. If $\rho_s \in \mathcal{Y}_M$ is a minimizer of $\mathcal{F}_{s,p}$ over $\mathcal{Y}_M$ \MMM for every $s\in(0,N/p)$\KKK, then we have 
		\[
		\lim_{s \to 0} \mathcal{F}_{s,p}(\rho_{s}) = 0 \qquad \mbox{ and } \qquad \lim_{s \to 0} \mathcal{D}_s = 0. 
		\]
	\end{proposition}
	\begin{proof}
		We recall that by Proposition \ref{prop:optimal-dilation}, we have 
		\[
		\limsup_{s \to 0} \mathcal{F}_{s,p}(\rho_{s}) \le \limsup_{s \to 0} \mu_s \le 0.
		\]
		We claim that 
		\begin{equation} \label{claim-1}
			\limsup_{s \to 0} \|\rho_{s}\|_{p'} < \infty.
		\end{equation}
		Since we have 
		\[
		\mathcal{F}_{s,p}(\rho_{s}) = -\frac{s\,p}{N-s\,p}\,\|\rho_{s}\|_{p'}^{p'},
		\]
		by Remark \ref{rmk:costante-el}, this would entail that 
		\[
		\liminf_{s \to 0} \mathcal{F}_{s,p}(\rho_{s}) = 0,
		\] 
		and so the desired result. 
		By recalling \eqref{extremal-functional-hls} and Corollary \ref{cor:invariant-functional}, we have that  
		\[
		H_s^{p'} \ge H^*_{p', s} = \frac{\|K_{s/2} \ast \rho_s\|_{p'}^{p'}}{M^{p'\frac{s}{N}}\, \|\rho_{s}\|^{p'\,\left(1-\frac{s}{N}\right)}_{p'}} = \frac{p^*_s}{p}\,\left(\frac{\|\rho_{s}\|_{p'}^{p'}}{M^{p'}} \right)^{\frac{s}{N}},
		\] 
		where $H_s$ is given as in Corollary \ref{cor:hls} and the last equality follows from \eqref{identita-pot-minimi}. Our claim \eqref{claim-1}, will follow by proving that 
		\begin{equation} \label{claim-2}
			\limsup_{s \to 0} H_s^{\frac{1}{s}} < \infty. 
		\end{equation}
		By \eqref{bound-costante-hls}, we have 
		\begin{equation} \label{quasi-finito}
			\begin{split}
				H_s^{\frac{1}{s}} &\le \mathcal{B}_s^{\frac{1}{s}}\,\left(\frac{N-s}{N}\right)^\frac{N-s}{N\,s}\,\left(\frac{\left(p'\right)^{\frac{N-s}{N}} + (p^*_s)^{\frac{N-s}{N}}}{(p^*_s)'\,p}\right)^\frac{1}{s} \\ &\le \mathcal{B}_s^{\frac{1}{s}}\,\left(\frac{N}{N-s\,p}\right)^{\frac{1}{s}} \le \mathcal{B}_s^{\frac{1}{s}}\,\left(1 + \frac{s\,p}{N-p}\right)^{\frac{1}{s}},
			\end{split}
		\end{equation}
		where we set 
		\[
		\mathcal{B}_s:= \frac{c_{N, s/2}}{s}\,N\,\omega_N^{\frac{N-s}{N}}.
		\]
		By \eqref{eqn:asym-behav}, we have $\lim_{s \to 0} \mathcal{B}_s = 1$, and we can write
		\[
		\begin{split}
			\mathcal{B}_s = \frac{1}{2}\, \frac{\pi^{-\frac{N}{2}}\,2^{-s}\,\Gamma\left(\frac{N-s}{2}\right)}{\Gamma\left(\frac{s}{2} +1\right)}\,N\,\left(\frac{\pi^{\frac{N}{2}}}{\Gamma\left(\frac{N}{2} +1\right)}\right)^{\frac{N-s}{N}},
		\end{split}
		\]
		for every \MMM $s \in (0, N/p)$. \KKK Since $\mathcal{B}_s$ is smooth in a neighborhood of $s= 0$, we have  
		\[
		\lim_{s \to 0} \frac{\mathcal{B}_s-1}{s} < \infty. 
		\]
		By spending this information in \eqref{quasi-finito}, we get \eqref{claim-2} which in turn implies \eqref{claim-1}. Eventually, by recalling Remark \ref{rmk:costante-el} we infer the desired asymptotic behavior also for $\mathcal{D}_s$. 
	\end{proof}
	\begin{proof}[\bf Proof of Theorem \ref{main2}]
		If $m > p'$, the claimed result is contained in Theorem \ref{thm:convergenza-forte-minimizzanti}. If $m = p'$, the conclusion follows by Theorem \ref{thm:m=p'}, \MMM Proposition \ref{lastprop} and Remark \ref{limitfair}. \KKK
	\end{proof}

	\begin{ack}
 The authors wish to warmly thank Bruno Volzone for fruitful discussions and Pedro Miguel Campos for having provided us the reference \cite{Kurokawa}. The authors are members of the {\it Gruppo Nazionale per l'Analisi Matematica, la Probabilit\`a
		e le loro Applicazioni} (GNAMPA) of the Istituto Nazionale di Alta Matematica (INdAM). F.B. is partially supported by the ``INdAM - GNAMPA Project Ottimizzazione Spettrale, Geometrica e Funzionale", CUP E5324001950001.  E.M. is supported by the MIUR- PRIN project 202244A7YL.
\end{ack}

\end{document}